\documentclass[11pt,reqno]{amsart}

\usepackage[T1]{fontenc}
\usepackage{lmodern}
\usepackage{microtype}
\usepackage[margin=1in]{geometry}
\usepackage{amsmath,amssymb,amsthm,mathtools}
\usepackage{booktabs}
\usepackage{enumitem}
\usepackage[dvipsnames]{xcolor}
\usepackage{pgfplots}
\usepackage[colorlinks=true,linkcolor=MidnightBlue,citecolor=BrickRed,
            urlcolor=RoyalBlue]{hyperref}
\usepgfplotslibrary{groupplots}
\usetikzlibrary{decorations.pathreplacing}
\pgfplotsset{compat=1.18}
\pgfplotsset{
  support diagonal/.style={black!45,thin,densely dotted,mark=none},
  support obstacle/.style={black,semithick,
    dash pattern=on 7pt off 2pt on 0.8pt off 2pt,mark=none},
  constrained lower/.style={BrickRed!85!black,very thick,solid,
    mark=o,mark repeat=120,mark size=1.15pt,
    mark options={solid,draw=black,fill=white,line width=0.45pt}},
  constrained upper/.style={MidnightBlue!90!black,very thick,
    dash pattern=on 5pt off 2pt,mark=square,mark repeat=120,
    mark phase=60,mark size=1.05pt,
    mark options={solid,draw=black,fill=white,line width=0.45pt}},
  ordinary lower/.style={ForestGreen!80!black,semithick,densely dotted,
    mark=triangle,mark repeat=20,mark size=1.15pt,
    mark options={solid,draw=black!65,fill=white}},
  ordinary upper/.style={ForestGreen!80!black,semithick,
    dash pattern=on 2.5pt off 1.4pt on 0.7pt off 1.4pt,
    mark=diamond,mark repeat=20,mark phase=10,mark size=1.05pt,
    mark options={solid,draw=black!65,fill=white}},
  support seam/.style={black!32,thin,
    dash pattern=on 1pt off 1.7pt,mark=none}
}

\numberwithin{equation}{section}
\allowdisplaybreaks

\newtheorem{theorem}{Theorem}[section]
\newtheorem{proposition}{Proposition}[section]
\newtheorem{lemma}{Lemma}[section]

\theoremstyle{definition}

\newtheorem{example}{Example}[section]
\theoremstyle{remark}
\newtheorem{remark}{Remark}[section]
\usepackage[capitalise,noabbrev]{cleveref}

\newcommand{\R}{\mathbb{R}}
\newcommand{\Pp}{\mathcal{P}}
\newcommand{\M}{\mathcal{M}}
\newcommand{\1}{\mathbf{1}}
\newcommand{\dd}{\,\mathrm{d}}
\newcommand{\Unif}{\operatorname{Unif}}
\newcommand{\supp}{\operatorname{supp}}
\newcommand{\lc}{\mathrm{lc}}

\newcommand{\KOneFiveDeltaB}{0.538575180}
\newcommand{\KOneFiveDeltaZero}{0.791971913}
\newcommand{\KOneFiveDeltaStar}{0.886555652}
\newcommand{\KOneFiveXZero}{-0.895985956}
\newcommand{\KTwoFiveEtaB}{1.261120319}

\newcommand{\CubicReferenceKTwoFive}{-1.446789916088}

\newcommand{\CubicReferenceKOneFive}{-0.673258643140}

\newcommand{\TraceLowerX}{0.0000000000}
\newcommand{\TraceLowerAlpha}{-0.8324843799}
\newcommand{\TraceLowerY}{-1.5000000000}
\newcommand{\TraceUpperX}{0.8500000000}
\newcommand{\TraceUpperAlpha}{-0.9252533586}
\newcommand{\TraceUpperY}{-0.6500000000}
\newcommand{\ComparisonTau}{0.7697191224}
\newcommand{\ComparisonEnd}{0.1553302625}

\newcommand{\ComparisonRows}{%
1.25 & -0.3710 & 0.3568 \\
1.50 & -0.6733 & 0.2596 \\
1.75 & -0.9348 & 0.1927 \\
2.00 & -1.1614 & 0.1426 \\
2.25 & -1.3379 & 0.1015 \\
2.50 & -1.4468 & 0.0653 \\
2.75 & -1.4928 & 0.0318 \\
3.00 & -1.5000 & 0.0000 \\
}

\title[A one-sided constrained martingale transport]
      {A one-sided constrained martingale transport between two uniform laws}
\author{Erhan Bayraktar}
\address[Erhan Bayraktar]{University of Michigan}
\email{erhan@umich.edu}
\author{Xin Zhang}
\address[Xin Zhang]{New York University}
\email{xz1662@nyu.edu}
\thanks{E.~Bayraktar is supported in part by the National Science
Foundation under grant DMS-2602036. X. Zhang is supported in part by the National Science
Foundation under grant DMS-2508556.}
\date{\today}


\begin{document}

\begin{abstract}
We minimize $\mathbb E[h(Y-X)]$ over martingale couplings of
$\Unif[-1,1]$ and $\Unif[-2,2]$ satisfying $Y\geq X-k$, where
$h\in C^1([-3,3])$ has convex derivative. Feasibility holds exactly for
$k\geq1$. For each such $k$, we construct a coupling that minimizes
all costs in this class. For $1<k<3$, its support consists of two graphs,
with $D(x)=x-k$ on $[k-2,1]$. The maps admit an explicit
parametrization for $2\leq k<3$ and are determined by scalar equations
with unique admissible roots for $1<k<2$. We prove optimality by a
dual inequality, using an analytic estimate in the latter range.
At $k=1$ the optimizer is $Y=X\pm1$ with equal probabilities;
for $k\geq3$ it is the ordinary left-curtain coupling.
In the unconstrained problem, left-monotonicity identifies the
left-curtain coupling, which is optimal for this cost class. Under the
constraint, a discrete example shows that the corresponding support
condition, even together with every two-source comparison, does not
imply optimality.
\end{abstract}

\maketitle
\enlargethispage{6pt}

\section{Introduction}\label{sec:introduction}

We consider the martingale transport problem
\[
 \inf\mathbb E[h(Y-X)],\qquad
 X\sim\Unif[-1,1],\quad Y\sim\Unif[-2,2],\quad
 \mathbb E[Y\mid X]=X,\quad Y\geq X-k,
\]
for $h\in C^1([-3,3])$ with convex derivative. Our main result constructs,
for every feasible $k$, a single coupling that minimizes all such costs.

The optimizer is described by maps $D(x)\leq x\leq U(x)$ and a
conditional law supported on $\{D(x),U(x)\}$. For $1<k<3$, set
$b=k-2$. The lower map is decreasing on $[-1,b]$ and satisfies
$D(x)=x-k$ on $[b,1]$; the upper map is increasing on $[-1,1]$.
We first determine the transport from $[b,1]$. Subtracting its two
marginals leaves finite measures that are coupled by a left-curtain
construction on $[-1,b]$. For $2\leq k<3$, the maps are explicitly
parametrized. For $1<k<2$, scalar equations determine the maps and the
source value at which $U(x)=1-k$. At $k=1$, the coupling
$Y=X\pm1$ with equal probabilities is optimal. For $k\geq3$, the
constraint permits the ordinary left-curtain coupling.

\Cref{thm:main} gives the common optimizer $\pi_k$.
\Cref{prop:feasibility} establishes the sharp feasibility threshold;
\cref{prop:unit-optimality,prop:ordinary} settle the endpoint regimes.
For $1<k<3$, the construction starts from the sources in $[b,1]$,
where the lower branch follows $y=x-k$. The upper branch solves a
terminal-value problem expressed through the principal Lambert $W$
function. After subtracting this transport, the residual source and
target measures are in convex order. Their left-curtain coupling is
explicit in one parameter for $2\leq k<3$
(\cref{prop:feasible-plan}) and is determined by scalar equations with
unique admissible roots for $1<k<2$ (\cref{prop:small-k}).

For smooth costs, we prove optimality by constructing bounded functions
$\phi,\psi,\theta$ such that
\[
 h(y-x)\geq\phi(x)+\psi(y)+\theta(x)(y-x),
 \qquad (x,y)\in[-1,1]\times[-2,2],\quad y\geq x-k,
\]
with equality at $y=D(x),U(x)$. Integration against any admissible
coupling eliminates the martingale term $\theta(x)(y-x)$ and gives a
lower bound attained by $\pi_k$.
The dual functions are built from probability measures $Q_x$ on
displacements. On $[b,1]$, these measures solve a Volterra equation;
the lower contact there requires a one-sided derivative inequality.
For $2\leq k<3$, a comparison in the target variable proves the dual
inequality (\cref{thm:optimality}). For $1<k<2$, the boundary point
$x-k$ may also lie on the upper graph. We combine a comparison in the
source variable with a sign estimate in the target variable
(\cref{thm:small-optimality}). The key estimate,
\cref{lem:small-dual-estimates}, follows from rational logarithm bounds
and explicit polynomial inequalities. Uniform approximation then
extends optimality to the full $C^1$ cost class.

\Cref{prop:Gamma-left} shows that $\pi_k$ is
$\Gamma_k$-left-monotone. However, \cref{ex:counter} shows that this
support condition, even together with every two-source comparison,
does not imply optimality. In the unconstrained problem,
left-monotonicity identifies the unique left-monotone coupling, which
minimizes every cost in the present class
\cite{BeiglbockJuillet,BeiglbockHenryLabordereTouzi,HenryLabordereTouzi}.
Uniqueness of the optimizer itself requires an additional strictness
condition; it is not asserted here. \Cref{rem:other-marginals}
discusses which parts of the construction extend beyond uniform
marginals.

The problem also has a robust pricing interpretation. If $X,Y$ are
centered and normalized discounted asset prices at two maturities,
call prices at all strikes determine their marginal laws
\cite{BreedenLitzenberger1978}, while the martingale condition links
the two maturities. The payoff $h(Y-X)$ depends on the price change.
The cost class includes the cubic payoff $(Y-X)^3$, the squared
positive-part payoffs $(Y-X-\tau)_+^2$, and exponential payoffs
$e^{\lambda(Y-X)}$ with $\lambda>0$. It excludes the forward-start
straddle $|Y-X|$. The constraint $Y-X\geq-k$ can be interpreted as a
prediction set expressing a bound on the downward price change
\cite{HouObloj2018,BartlKupperNeufeld2020}. The value is then the
smallest model price consistent with the marginal laws, the martingale
condition, and that bound. The dual inequality provides the
corresponding subhedge. Conditioning an unconstrained model on
$\{Y-X\geq-k\}$ generally preserves neither the marginals nor the
martingale property, so the constrained coupling must be constructed
directly.

Strassen's theorem characterizes martingale feasibility by convex order
\cite{Strassen}. For an atomless source law on the line, the left-curtain
coupling is supported on at most two graphs \cite{BeiglbockJuillet} and
is optimal for martingale Spence--Mirrlees costs
\cite{BeiglbockHenryLabordereTouzi,HenryLabordereTouzi}.
Constructions by potential functions and extensions to shadow couplings
appear in \cite{HobsonNorgilas,BeiglbockJuilletShadow2021}.
Related potential methods construct supermartingale couplings for
marginals in convex-decreasing order
\cite{BayraktarDengNorgilas2023,BayraktarDengNorgilas2024}.
Quasi-sure duality and dual attainment are established in
\cite{BeiglbockNutzTouzi2017}. Connections with robust hedging and stochastic control are
studied in \cite{BeiglbockHenryLaborderePenkner,
GalichonHenryLabordereTouzi2014}. Sharp bounds for forward-start straddles
are obtained in \cite{HobsonNeuberger2012,HobsonKlimmek2015}.

For domain constraints, \cite{BayraktarZhangZhou} proves existence,
duality, and monotonicity results. Related formulations with linear
constraints and additional market information appear in
\cite{Zaev,AcciaioCoxHuesmann}. These results provide feasibility and
necessary optimality criteria; here we determine the coupling explicitly
and verify the dual inequality on the entire admissible domain.
The formulas also give examples for testing discrete approximations,
such as those developed for ordinary martingale transport in
\cite{GuoObloj2019}.

\Cref{sec:problem} states the main result and proves the endpoint cases.
\Cref{sec:construction,sec:optimality-limitations} give the construction
and optimality proof. \Cref{sec:numerics,sec:discussion} contain numerical
examples and a summary of the parameter regimes.

\section{The problem and main result}\label{sec:problem}

Let
\begin{equation*}
  \mu(\dd x)=\frac12\1_{[-1,1]}(x)\dd x,
  \qquad
  \nu(\dd y)=\frac14\1_{[-2,2]}(y)\dd y.
\end{equation*}
Here $\Pp(S)$ denotes the Borel probability measures on a Polish space
$S$, and $\pi_i$ is the $i$th marginal of $\pi\in\Pp(\R^2)$, $i=1,2$.
For a coupling with first marginal $\mu$, we write
$\pi(\dd x,\dd y)=\mu(\dd x)\pi_x(\dd y)$ for a Borel disintegration;
the conditional laws $\pi_x$ are defined up to a $\mu$-null set. We write
$\mu\preceq_{\mathrm{cx}}\nu$ when
$\int f\dd\mu\leq\int f\dd\nu$ for every convex $f$ for which the
integrals are finite. The coupling at $k=1$ constructed below shows
that $\mu\preceq_{\mathrm{cx}}\nu$.
For $k>0$, define
\begin{equation*}
  \Gamma_k:=\{(x,y)\in\R^2:y\geq x-k\}
\end{equation*}
and
\begin{equation*}
 \M_k(\mu,\nu):=
 \left\{\pi\in\Pp(\R^2):
 \begin{array}{l}
   \pi_1=\mu,\quad \pi_2=\nu,\quad \pi(\Gamma_k)=1,\\[2pt]
   \displaystyle\int y\,\pi_x(\dd y)=x
       \quad\text{for }\mu\text{-almost every }x
 \end{array}\right\}.
\end{equation*}
All admissible displacements belong to $[-3,3]$.  For
$h\in C^1([-3,3])$, define
\begin{equation}\label{eq:primal}
 V_k(h):=\inf_{\pi\in\M_k(\mu,\nu)}
           \int_{\R^2}h(y-x)\,\pi(\dd x,\dd y).
\end{equation}
When $\M_k(\mu,\nu)=\varnothing$ we use the convention
$V_k(h)=+\infty$. Every continuous cost on $[-3,3]$ is bounded,
so the objective is finite for every admissible coupling. For
$h\in C^3([-3,3])$, convexity of $h'$ is equivalent to
$c_{xyy}(x,y)=-h^{(3)}(y-x)\leq0$ for $c(x,y)=h(y-x)$, the
martingale Spence--Mirrlees sign for minimization.

\begin{theorem}[A cost-independent optimal coupling]\label{thm:main}
The set $\M_k(\mu,\nu)$ is nonempty if and only if $k\geq1$.
For every $k\geq1$, there is a two-graph coupling $\pi_k\in\M_k(\mu,\nu)$
such that
\[
 V_k(h)=\int h(y-x)\,\pi_k(\dd x,\dd y)
 \quad\text{for every }h\in C^1([-3,3])\text{ with convex derivative}.
\]
At $k=1$, it is given by $Y=X\pm1$ with equal conditional probabilities,
and $V_1(h)=[h(-1)+h(1)]/2$, as proved in \cref{prop:unit-optimality}.
For $k>1$, its construction has three regimes:
\begin{enumerate}[label=\textup{(\roman*)},leftmargin=2.2em]
\item \label{item:smallk}
      For $1<k<2$, the maps are determined by the scalar equations
      in \cref{sec:small-k}, each with a unique admissible solution.
\item \label{item:middle}
      For $2\leq k<3$, the maps have the closed parametrization in
      \cref{sec:explicit} and agree with the ordinary left-curtain
      maps on $[-1,2k-5]$.
\item \label{item:k3}
      For $k\geq3$, $\pi_k$ is the ordinary left-curtain coupling in
      \cref{prop:ordinary}.
\end{enumerate}
For $1<k<3$, the lower map decreases on $[-1,b]$ and follows
$D(x)=x-k$ on $[b,1]$, where $b=k-2$. The upper map increases throughout.
\end{theorem}

The theorem asserts the existence of a common optimizer, not uniqueness
for each cost. For example, every admissible coupling has the same
quadratic displacement cost, since
$\mathbb E[(Y-X)^2]=\mathbb E[Y^2]-\mathbb E[X^2]=1$.
Part~\ref{item:smallk} is proved in \cref{thm:small-optimality} using
\cref{lem:small-dual-estimates}.

It suffices to prove optimality for $C^3$ costs with nonnegative third
derivative. Indeed, for $g=h'$ let $g_n$ be the Bernstein polynomial
on $[-3,3]$:
\[
 g_n(z)=\sum_{j=0}^n g\!\left(-3+\frac{6j}{n}\right)
     \binom{n}{j}t^j(1-t)^{n-j},\qquad t=\frac{z+3}{6}.
\]
For $n\geq2$, convexity of $g$ makes all second differences of the sampled
values nonnegative, so $g_n''\geq0$ on $[-3,3]$. Bernstein approximation
gives $g_n\to g$ uniformly. Hence the polynomials
\begin{equation}\label{eq:convex-approximation}
 h_n(z):=h(0)+\int_0^z g_n(s)\dd s
 \quad\text{satisfy}\quad
 h_n'''\geq0,\qquad
 \|h_n-h\|_\infty\leq3\|g_n-h'\|_\infty\longrightarrow0.
\end{equation}
All norms are on $[-3,3]$. Since
$|\int(h_n-h)(y-x)\dd\pi|\leq\|h_n-h\|_\infty$ for every probability
coupling $\pi$, optimality of the same $\pi_k$ passes to the limit.

\subsection{Feasibility and the endpoint cases}

We identify the feasible values of $k$ and determine optimal couplings
at $k=1$ and for $k\geq3$. This reduces the remaining construction and
optimality proofs to $1<k<3$.

\begin{proposition}[Sharp feasibility threshold]\label{prop:feasibility}
The set $\M_k(\mu,\nu)$ is nonempty exactly when $k\geq1$.
\end{proposition}

\begin{proof}
If $k<1$, every admissible pair satisfies
\[
  Y\geq X-k\geq-1-k>-2.
\]
This is incompatible with the fact that $\nu$ assigns positive mass to
$[-2,-1-k)$. Conversely, at $k=1$ let $Z$ be independent of $X$ with
$\mathbb P(Z=-1)=\mathbb P(Z=1)=1/2$, and put $Y=X+Z$. Then
$\mathbb E[Y\mid X]=X$ and $Y\geq X-1$. The law of $Y$ is the mixture
of $\Unif[-2,0]$ and $\Unif[0,2]$, each with weight $1/2$, hence is
$\Unif[-2,2]$. The same coupling is feasible for every $k\geq1$.
\end{proof}

\paragraph{$\Gamma_k$-convex order.}
For $x\in[-1,1]$, let
$I_k(x):=[-2,2]\cap[x-k,\infty)$ be the permitted target interval.
For $f\in C([-2,2])$, define its \emph{$\Gamma_k$-convex envelope} by
\[
 f^{\Gamma_k}(x):=
 \inf\left\{\int f(y)Q(\dd y):
       Q\in\Pp(I_k(x)),\ \int yQ(\dd y)=x\right\}.
\]
This is the lower convex envelope of $f$ restricted to $I_k(x)$,
evaluated at $x$. The infimum is attained by a probability measure
supported on at most two points.

We say that $\mu$ precedes $\nu$ in \emph{$\Gamma_k$-convex order}, and
write $\mu\preceq_{\Gamma_k}\nu$, if
\[
 \int f^{\Gamma_k}(x)\mu(\dd x)\leq\int f(y)\nu(\dd y)
 \qquad\text{for every }f\in C([-2,2]).
\]
For the present compact constraint, the constrained Strassen theorem
\cite[Proposition~2.1]{BayraktarZhangZhou} gives
\[
 \M_k(\mu,\nu)\ne\varnothing
 \quad\Longleftrightarrow\quad \mu\preceq_{\Gamma_k}\nu.
\]
Ordinary convex order is insufficient: $\mu\preceq_{\mathrm{cx}}\nu$
holds even for $k<1$. The constrained order characterizes feasibility
but does not identify a cost minimizer. The measures attaining
$f^{\Gamma_k}(x)$ depend on $f$ and need not combine to have second
marginal $\nu$. The two-source optimality criterion and its limitations
are discussed in \cref{sec:gamma-monotonicity}.

\begin{proposition}[Optimality at $k=1$]\label{prop:unit-optimality}
Let $h\in C^1([-3,3])$ have convex derivative. The coupling
$Y=X\pm1$ with equal conditional probabilities minimizes
\eqref{eq:primal} at $k=1$, and
\[
 V_1(h)=\frac{h(-1)+h(1)}2.
\]
\end{proposition}

\begin{proof}
For any $\pi\in\M_1(\mu,\nu)$, put $Z=Y-X$. The constraint,
martingale property, and prescribed marginals give
\[
 -1\leq Z\leq3,\qquad
 \mathbb E^\pi Z=0,\qquad
 \mathbb E^\pi Z^2=\mathbb E^\pi Y^2-\mathbb E^\pi X^2=1.
\]
First suppose $h\in C^3$ with $h'''\geq0$. Let $P_h$ be the quadratic
polynomial that agrees with $h$ at $-1$ and $1$ and with $h'$ at $1$:
\[
 P_h(z)=h(1)+h'(1)(z-1)
       +\frac{h(-1)-h(1)+2h'(1)}4(z-1)^2.
\]
The Hermite interpolation remainder gives, for $z\notin\{-1,1\}$,
\[
 h(z)-P_h(z)=\frac{h'''(\xi_z)}6(z+1)(z-1)^2\geq0,
 \qquad -1\leq z\leq3,
\]
where $\xi_z$ lies between the smallest and largest of $-1,1,z$.
Equality holds at the two interpolation points. Since
$\mathbb E^\pi(Z-1)=-1$ and $\mathbb E^\pi(Z-1)^2=2$, it follows that
\[
 \mathbb E^\pi h(Z)\geq\mathbb E^\pi P_h(Z)
 =\frac{h(-1)+h(1)}2.
\]
The feasible coupling in \cref{prop:feasibility} attains this bound.
For general $h\in C^1$ with convex derivative, apply the bound to
$h_n$ from \eqref{eq:convex-approximation} and pass to the uniform limit.
\end{proof}

\begin{proposition}[The ordinary left-curtain regime]\label{prop:ordinary}
Define
\begin{equation}\label{eq:ordinary-maps}
 D_{\lc}(x)=-\frac{x}{2}-\frac32,
 \qquad
 U_{\lc}(x)=\frac{3x}{2}+\frac12,
 \qquad
 q_{\lc}(x)=\frac34.
\end{equation}
The kernel
\[
 K_{\lc}(x,\dd y)
 =(1-q_{\lc})\delta_{D_{\lc}(x)}(\dd y)
   +q_{\lc}\delta_{U_{\lc}(x)}(\dd y)
\]
defines the left-curtain coupling of $\mu$ and $\nu$.  It is feasible for
$\Gamma_k$ if and only if $k\geq3$ and, in that range, minimizes
\eqref{eq:primal} whenever $h'$ is convex.
\end{proposition}

\begin{proof}
For $x>-1$ one has
$D_{\lc}(x)<x<U_{\lc}(x)$ and
$(1-q_{\lc})D_{\lc}+q_{\lc}U_{\lc}=x$.  The lower branch carries source
density $(1-q_{\lc})/2=1/8$ and has derivative $-1/2$, so its image has
density $1/4$ on $[-2,-1]$.  The upper branch carries density $3/8$ and
has derivative $3/2$, so its image has density $1/4$ on $[-1,2]$.  Thus
the second marginal is $\nu$. Since $D_{\lc}$ decreases and
$U_{\lc}$ increases, the support is left-monotone. The uniqueness of
the left-monotone coupling identifies it as the left-curtain coupling
\cite{BeiglbockJuillet}. Moreover,
\begin{equation*}
 x-D_{\lc}(x)=\frac32(x+1)\leq3.
\end{equation*}
Thus the coupling is feasible when $k\geq3$. If $k<3$, its lower
branch violates the constraint for $x>2k/3-1$. This source interval has
positive $\mu$-mass and the lower branch has weight $1/4$, proving
infeasibility in that range.
For smooth $h$ with $h'''>0$, the cost
$c(x,y)=h(y-x)$ satisfies $c_{xyy}=-h'''<0$, and the martingale
Spence--Mirrlees theorem selects the left-curtain coupling for the
minimization convention used here
\cite{BeiglbockHenryLabordereTouzi,HenryLabordereTouzi}.
For smooth $h$ with $h'''\geq0$, apply this result to
$h_\varepsilon(z)=h(z)+\varepsilon z^3$ and let
$\varepsilon\downarrow0$. Uniform convergence preserves the optimality
inequality. The approximation in \eqref{eq:convex-approximation} gives
the $C^1$ statement. Since the unconstrained minimizer is feasible for
$k\geq3$, it also minimizes the constrained problem.
\end{proof}

\section{Construction of the couplings}\label{sec:construction}

For $1<k<3$, set $b=k-2$. We construct continuous maps
$D,U:[-1,1]\to[-2,2]$ satisfying $D(x)\leq x\leq U(x)$ and
$D(x)=x-k$ on $[b,1]$. The martingale weights are
\[
 q(x)=\frac{x-D(x)}{U(x)-D(x)},\qquad p(x)=1-q(x),
 \qquad x>-1.
\]
After determining the transport from $[b,1]$, we subtract its marginals
from $\mu$ and $\nu$. We couple the resulting measures
$\bar\mu,\bar\nu$ by a martingale transport $\bar\pi$ satisfying
\[
 (\operatorname{pr}_2)_\#
 \bigl(\bar\pi|_{[-1,x]\times\mathbb R}\bigr)
 =\bar\nu|_{[D(x),U(x)]},\qquad -1\leq x\leq b,
\]
where $\operatorname{pr}_2(x,y)=y$, $D$ decreases, and $U$ increases
on $[-1,b]$.
Equality of mass and first moment determines these endpoints. The
formulas depend on whether $D(x)$ and $U(x)$ lie below or above $1-k$,
where the density of $\bar\nu$ changes.

\subsection{Transport from \texorpdfstring{$[b,1]$}{[b,1]} and residual measures}\label{sec:tail}

We determine the transport from $[b,1]$ and subtract its marginals from
$\mu,\nu$. We then derive the mass and first-moment equations that
determine $D,U$ on $[-1,b]$.

On $[b,1]$, impose
\begin{equation*}
  D(x)=x-k,\qquad U(x)>x.
\end{equation*}
The martingale condition then gives
\begin{equation*}
  q(x)=\frac{k}{U(x)-x+k},
  \qquad
  p(x):=1-q(x)=\frac{U(x)-x}{U(x)-x+k}.
\end{equation*}
Requiring the transport along $U$ to have target density $1/4$ gives
$U'(x)/4=q(x)/2$, hence
\begin{equation}\label{eq:tail-ode}
  U'(x)=\frac{2k}{U(x)-x+k},
  \qquad U(1)=2.
\end{equation}
Integration yields
\begin{equation}\label{eq:tail-implicit}
 U(x)+2k\log\bigl(k-U(x)+x\bigr)
 =2+2k\log(k-1).
\end{equation}

Let
$W_0:[-e^{-1},\infty)\to[-1,\infty)$ denote the principal real branch of
the Lambert function, characterized by $W_0(z)e^{W_0(z)}=z$.  Setting
$r(x):=k-U(x)+x$ in \eqref{eq:tail-implicit} gives
\[
 r(x)e^{-r(x)/(2k)}
 =(k-1)\exp\left(\frac{2-x-k}{2k}\right).
\]
Thus
\begin{equation}\label{eq:tail-lambert}
 U(x)=x+k+2kW_0\!\left(
 -\frac{k-1}{2k}\exp\left(\frac{2-x-k}{2k}\right)
 \right),\qquad b\leq x\leq1.
\end{equation}
For $\xi_k(x):=-\frac{k-1}{2k}\exp\left(\frac{2-x-k}{2k}\right)$ and $b\leq x\leq1$,
\[
 |\xi_k(x)|\leq\frac{k-1}{2k}\exp\left(\frac{2-k}{k}\right)
 \leq\frac14<e^{-1};
\]
the middle bound is maximized at $k=2$. Since
$1/4<(1/2)e^{-1/2}$ and $W_0$ is increasing, one has
$W_0(\xi_k(x))>-1/2$, so $U(x)-x>0$. The other real branch,
$W_{-1}:[-e^{-1},0)\to(-\infty,-1]$, would give $r=-2kW_{-1}(\xi_k(x))\geq2k$ and hence
$U-x=k-r<0$, contrary to $U\geq x$. In particular, $\eta_b:=U(b)$ satisfies
\begin{equation*}
 \eta_b=2k-2+2kW_0\!\left(
 -\frac{k-1}{2k}\exp\left(\frac{2-k}{k}\right)
 \right).
\end{equation*}

Set
\begin{equation}\label{eq:E-rho}
 E(\eta):=(k-1)\exp\left(\frac{2-\eta}{2k}\right),
 \qquad
 \rho(\eta):=E(\eta)+\eta-k.
\end{equation}
Equation \eqref{eq:tail-implicit} is equivalent to
\begin{equation}\label{eq:tail-param}
 x=\rho(\eta),\qquad U(x)=\eta,\qquad
 k-U(x)+x=E(\eta).
\end{equation}
The function $E$ decreases on $[b,2]$ and satisfies $E(b)<k$. To prove
the latter inequality, define
\[
 f(k):=\log\frac{k}{k-1}-\frac{4-k}{2k}.
\]
Then $f'(k)=(k-2)/(k^2(k-1))$, so its minimum on $(1,3)$ is
$f(2)=\log2-1/2>0$. Hence $\rho'(\eta)=1-E(\eta)/(2k)>0$ there. Moreover,
$\rho(b)<b$ and $\rho(2)=1>b$. Therefore $\eta_b$ is equivalently the unique solution in $(b,2)$ of
\begin{equation}\label{eq:eta-b}
  \rho(\eta_b)=b=k-2.
\end{equation}
For $\eta\in[\eta_b,2]$, the displacement and weights are
\begin{equation*}
 w(\eta):=U-x=k-E(\eta),\qquad
 q=\frac{k}{2k-E(\eta)},\qquad
 p=\frac{k-E(\eta)}{2k-E(\eta)}.
\end{equation*}
At $k=3$, set $\eta_b=2$. The interval $[b,1]$ reduces to a point,
and the construction below gives the ordinary left-curtain coupling.
Formulas at degenerate intervals are understood by continuity.

\begin{lemma}[The terminal-value problem]
Let $1<k<3$ and $b=k-2$.  The terminal-value problem
\eqref{eq:tail-ode} has a unique $C^1$ solution on $[b,1]$ satisfying
$U(x)>x$.  It is given by \eqref{eq:tail-lambert}; moreover,
\[
 0<U(x)-x<k,\qquad U'(x)>0,\qquad U(b)=\eta_b.
\]
\end{lemma}

\begin{proof}
The estimates above show that
$\rho:[\eta_b,2]\to[b,1]$ is a strictly increasing $C^1$ bijection.
Define $U=\rho^{-1}$.  Since $x=E(U)+U-k$, one has
$U-x=k-E(U)\in(0,k)$ and
\[
 U'(x)=\frac{1}{\rho'(U(x))}
       =\frac{2k}{2k-E(U(x))}
       =\frac{2k}{U(x)-x+k}.
\]
The endpoint values follow from $\rho(\eta_b)=b$ and $\rho(2)=1$,
and solving for the inverse gives
\eqref{eq:tail-lambert}.  Finally, for $w=U-x$ the equation becomes
$w'=(k-w)/(k+w)$.  Its vector field is locally Lipschitz for $w>-k$;
therefore any solution satisfying $U>x$ and $U(1)=2$ agrees with the
one just constructed on the whole interval.
\end{proof}

\medskip\noindent\emph{Residual measures.}

The intervals $D([b,1])=[-2,1-k]$ and $U([b,1])=[\eta_b,2]$ are
disjoint. Indeed, $\eta_b>1-k$ follows from $\eta_b>b$ if $1-k\leq b$;
otherwise $1-k\in(b,2)$ and
\[
 \rho(1-k)-b=(k-1)\left(\exp\!\left(\frac{k+1}{2k}\right)-3\right)<0,
\]
since $\exp((k+1)/(2k))<e<3$. Thus $\rho(1-k)<\rho(\eta_b)$, which gives the same inequality.

The transport from $[b,1]$ has target density $1/4$ on $[\eta_b,2]$
and $p(y+k)/2$ on $[-2,1-k]$. Its removal leaves
\begin{equation*}
 \bar\mu(\dd x)=\frac12\1_{[-1,b]}(x)\dd x
\end{equation*}
and $\bar\nu(\dd y)=\bar r(y)\dd y$, where
\begin{equation*}
 \bar r(y)=
 \begin{cases}
  \displaystyle\frac14-\frac12p(y+k),&-2\leq y\leq1-k,\\[6pt]
  \displaystyle\frac14,&1-k<y\leq\eta_b,\\[4pt]
 0,&\text{otherwise}.
 \end{cases}
\end{equation*}
Since $0<p<1/2$, one has $0<\bar r<1/4$ on $[-2,1-k]$ and
$\bar r=1/4$ on $(1-k,\eta_b]$. The finite measures $\bar\mu$ and
$\bar\nu$ both have mass $(b+1)/2$. They have equal first moments
because the removed transport is a martingale coupling of its marginals.

\begin{lemma}[Residual convex order]\label{lem:residual-convex-order}
For every $1<k<3$, the finite measures $\bar\mu$ and $\bar\nu$ have equal
mass and first moment and satisfy
\[
 \int f\dd\bar\mu\leq\int f\dd\bar\nu
\]
for every convex function $f$ on $[-2,2]$.
\end{lemma}

\begin{proof}
Let $I=[-1,b]$ and
$\sigma=\bar\mu-\bar\nu$.  On $I$, the density of $\bar\mu$ is $1/2$
whereas that of $\bar\nu$ is at most $1/4$, so $\sigma|_I\geq0$.
Outside $I$ one has $\sigma\leq0$.  Let $\ell$ be the affine function
interpolating $f$ at $-1$ and $b$.  Convexity gives $f\leq\ell$ on $I$
and $f\geq\ell$ outside $I$.  Therefore
\[
 \int(f-\ell)\dd\sigma\leq0.
\]
Equality of mass and first moment gives
$\int\ell\dd\sigma=0$, and hence $\int f\dd\bar\mu\leq
\int f\dd\bar\nu$.
\end{proof}

For $-1\leq x\leq b$, the prescribed marginal identity for
$\bar\pi|_{[-1,x]\times\mathbb R}$ requires
$-2\leq D(x)\leq x\leq U(x)\leq\eta_b$ and
\begin{align}
 \int_{D(x)}^{U(x)}\bar r(y)\dd y
   &=\frac{x+1}{2}, \label{eq:residual-mass}\\
 \int_{D(x)}^{U(x)}y\bar r(y)\dd y
   &=\frac{x^2-1}{4},\qquad -1\leq x\leq b.
   \label{eq:residual-moment}
\end{align}
Where $D,U$ are differentiable, $D<U$, and both endpoints avoid $1-k$,
these identities give the left-curtain mass balances
\cite{HenryLabordereTouzi,BeiglbockJuillet}:
\begin{align*}
 \bar r(U(x))U'(x)
   &=\frac12\frac{x-D(x)}{U(x)-D(x)},\\
 -\bar r(D(x))D'(x)
   &=\frac12\frac{U(x)-x}{U(x)-D(x)}.
\end{align*}
At an endpoint crossing $1-k$, derivatives are taken one-sided. The
identities are not asserted at $D(-1)=U(-1)=-1$.
At $x=b$, the mass equation requires $[D(b),U(b)]$ to contain all of
$\bar\nu$. Since $\bar r>0$ almost everywhere on $[-2,\eta_b]$,
\begin{equation*}
  D(b)=-2,\qquad U(b)=\eta_b.
\end{equation*}
These values agree with the maps on $[b,1]$.
For each $-1<x\leq b$, the mass and first moment determine the endpoints
uniquely. Indeed, let $Q$ be the inverse of the cumulative mass
function of $\bar\nu$, and set $m=(x+1)/2$. An interval of mass $m$
beginning at cumulative mass $s$ has first moment
$\int_s^{s+m}Q(u)\dd u$. This expression is strictly increasing in
$s$, because $\bar r>0$ throughout the interior of its support.
At $x=-1$, the ordering $D\leq x\leq U$ and zero interval mass force
$D(-1)=U(-1)=-1$.
Propositions~\ref{prop:feasible-plan} and~\ref{prop:small-k} construct
such residual couplings. Their decreasing lower and increasing upper
maps imply left-monotonicity; the uniqueness theorem of
\cite{BeiglbockJuillet}, applied after normalization, identifies each
with the left-curtain coupling of $\bar\mu$ and $\bar\nu$.

\subsection{Construction for
\texorpdfstring{$2\leq k\leq3$}{2 <= k <= 3}}
\label{sec:explicit}

We give an explicit parametrization of $D,U$ on $[-1,b]$ for
$2\leq k\leq3$ and verify that the resulting coupling belongs to
$\M_k(\mu,\nu)$.

Assume $2\leq k\leq3$ and put
\begin{equation*}
  a:=2k-5.
\end{equation*}
On $[-1,a]$, both endpoints lie in $[1-k,\eta_b]$, where
$\bar r=1/4$ almost everywhere. The maps therefore equal \eqref{eq:ordinary-maps}, with
\begin{equation*}
 D(a)=1-k,\qquad U(a)=3k-7.
\end{equation*}

\medskip\noindent\emph{Maps on $[a,b]$.}

Define
\begin{equation*}
  R(\eta):=\sqrt{k^2-(\eta+2)E(\eta)},
\end{equation*}
with $E$ from \eqref{eq:E-rho}. For $\eta\in[\eta_b,2]$ set
\begin{align}
 x_m(\eta)&:=E(\eta)-2+R(\eta),\label{eq:x-middle}\\
 D_m(\eta)&:=E(\eta)+\eta-2k,\label{eq:D-middle}\\
 U_m(\eta)&:=E(\eta)-2+2R(\eta).\label{eq:U-middle}
\end{align}
The following results show that $R$ is real and $x_m$ decreases from
$b$ to $a$ as $\eta$ increases from $\eta_b$ to $2$. Thus
\eqref{eq:D-middle}--\eqref{eq:U-middle} define $D,U$ on $[a,b]$ by
inverting \eqref{eq:x-middle}.

\begin{lemma}[Nonnegativity of $R^2$]\label{lem:R-positive}
For $2\leq k\leq3$ and $\eta\in[\eta_b,2]$,
\[
 k^2-(\eta+2)E(\eta)\geq0.
\]
The inequality is strict except at the degenerate point $(k,\eta)=(2,2)$.
\end{lemma}

\begin{proof}
Let $F(\eta)=k^2-(\eta+2)E(\eta)$.  Since $E'=-E/(2k)$,
\[
 F'(\eta)=\frac{E(\eta)}{2k}(\eta+2-2k)\leq0
 \qquad (\eta\leq2,\ k\geq2).
\]
Moreover, $F(2)=(k-2)^2$.  Hence $F(\eta)\geq F(2)\geq0$, with the
stated equality case.
\end{proof}

\begin{proposition}[Continuity, monotonicity, and the constraint]\label{prop:matching}
For $2\leq k\leq3$, the maps defined by
\eqref{eq:ordinary-maps}, \eqref{eq:x-middle}--\eqref{eq:U-middle}, and
\eqref{eq:tail-param} have the following properties.
\begin{enumerate}[label=\textup{(\alph*)},leftmargin=2em]
\item They match continuously at $x=a$ and $x=b$.
\item The lower map is strictly decreasing on $[-1,b]$.
\item The upper map is strictly increasing on $[-1,1]$.
\item The upper displacement $w(x)=U(x)-x$ is nondecreasing.
\item The constraint $D(x)\geq x-k$ holds on $[-1,b]$, with
      equality only at $x=b$ when $k<3$.
\item One has $D(-1)=U(-1)=-1$ and $D(x)<x<U(x)$ for every
      $-1<x\leq1$.
\end{enumerate}
\end{proposition}

\begin{proof}
At $\eta=2$, one has $E(2)=k-1$ and $R(2)=k-2$. Hence
\[
 x_m(2)=2k-5=a,\quad D_m(2)=1-k,\quad U_m(2)=3k-7,
\]
which proves matching at $x=a$.

At $\eta=\eta_b$, equation \eqref{eq:eta-b} says
$E(\eta_b)+\eta_b=2k-2$. Consequently,
\[
 R(\eta_b)^2-(k-E(\eta_b))^2
 =E(\eta_b)\bigl(2k-2-\eta_b-E(\eta_b)\bigr)=0.
\]
The positive root is $R(\eta_b)=k-E(\eta_b)$, and therefore
\[
 x_m(\eta_b)=b,\qquad D_m(\eta_b)=-2,\qquad
 U_m(\eta_b)=\eta_b.
\]

On the interior of every nondegenerate parameter interval,
$E'=-E/(2k)$ and
\begin{equation*}
 R'(\eta)=\frac{E(\eta)(\eta+2-2k)}{4kR(\eta)},\qquad
 x_m'(\eta)=
 \frac{E(\eta)(\eta+2-2k-2R(\eta))}{4kR(\eta)}<0.
\end{equation*}
Moreover $D_m'(\eta)=1-E(\eta)/(2k)>0$, so $D$ decreases as $x$
increases. Since $U=x+R$, its displacement satisfies
\begin{equation*}
 \frac{\dd w}{\dd x}
 =\frac{\eta+2-2k}{\eta+2-2k-2R(\eta)}\geq0.
\end{equation*}
Thus $U'=1+w'>0$. On $[-1,a]$, $w=(x+1)/2$. On $[b,1]$, $w=k-E(\eta)$ and
\[
 \frac{\dd w}{\dd x}=\frac{E(\eta)}{2k-E(\eta)}>0.
\]
On $[-1,a]$,
$D_{\lc}(x)-(x-k)=k-\frac32(x+1)\geq6-2k\geq0$. On $[a,b]$,
\[
 D_m-x_m+k=\eta-k+2-R,\qquad
 (\eta-k+2)^2-R^2=(\eta+2)(\rho(\eta)-b)\geq0.
\]
Since $\eta-k+2>0$, this proves $D\geq x-k$, with equality at
$\eta=\eta_b$, or $x=b$. Finally, $D_{\lc}\leq x$ on $[-1,a]$, while
on $[a,b]$
$x_m-D_m=R+2k-2-\eta\geq0$; also $U-x=w\geq0$ throughout.
The values at the degenerate intervals for $k=2$ and $k=3$ follow by
continuity; the derivative quotients are asserted only where defined.
\end{proof}

\medskip\noindent\emph{Differential equation on $[a,b]$.}

Let
\begin{equation*}
  P(s):=\int_s^1p(r)\dd r,\qquad b\leq s\leq1,
\end{equation*}
where $p$ is the lower weight for sources in $[b,1]$. For $a\leq x\leq b$,
the residual mass equation gives
\begin{equation}\label{eq:U-in-D}
 U(x)=2+2x+D(x)+2P(D(x)+k).
\end{equation}
For $a<x<b$, differentiation using $P'=-p$ yields
\begin{equation*}
 U'(x)=2+\bigl(1-2p(D(x)+k)\bigr)D'(x).
\end{equation*}
Since $\bar r(U)=1/4$, the upper mass balance also gives
\begin{equation*}
 U'(x)=2q(x)=2\frac{x-D(x)}{U(x)-D(x)}.
\end{equation*}
Using $U-D=2+2x+2P(D+k)$ gives
\begin{equation}\label{eq:corrected-ode}
 \boxed{
 D'(x)=
 -\frac{4+2x+2D(x)+4P(D(x)+k)}
 {(1-2p(D(x)+k))(2+2x+2P(D(x)+k))}.}
\end{equation}
For $2<k<3$, the initial condition is $D(a)=1-k$. At $k=2$,
$a=-1$ and $U(a)=D(a)=-1$, so the right-hand side is undefined there.
The parametrization gives the continuous endpoint value and solves the
equation on $(-1,0]$. At $k=3$, $a=b=1$.
The integral in \eqref{eq:corrected-ode} can be evaluated from
\begin{equation*}
  P(s)=\frac12U_{\mathrm a}(s)-s,
\end{equation*}
where $U_{\mathrm a}$ denotes the upper map on $[b,1]$ defined by
\eqref{eq:tail-ode}.  Indeed, both sides vanish at $s=1$, while the
derivative of the right-hand side is $U_{\mathrm a}'/2-1=q-1=-p$.

\begin{lemma}[Solution of the differential equation]
For $2\leq k<3$, the parametrized maps
$x_m,D_m,U_m$ satisfy \eqref{eq:U-in-D} on $[a,b]$ and
\eqref{eq:corrected-ode} on $(a,b)$.  If $2<k<3$, they give the unique $C^1$ solution
of the initial-value problem $D(a)=1-k$ that remains in the admissible
region
\[
 b\leq D(x)+k\leq1,\qquad U(x)>D(x).
\]
For $k=2$, they give a $C^1$ solution on $(-1,0]$ with a continuous
extension to $(D,U)(-1)=(-1,-1)$.
\end{lemma}

\begin{proof}
Put $s=D_m(\eta)+k$.  The definitions give
$s=E(\eta)+\eta-k=\rho(\eta)$, so the parametrization on $[b,1]$ and the
identity for $P$ imply
\[
 U_{\mathrm a}(s)=\eta,\qquad
 P(s)=\frac{\eta}{2}-\rho(\eta).
\]
Substitution in the right-hand side of \eqref{eq:U-in-D} gives $U_m$.
Writing $C=2k-2-\eta$, direct differentiation also gives
\[
 U_x=2\frac{R+C}{2R+C}
     =2\frac{x-D}{U-D}.
\]
Differentiating \eqref{eq:U-in-D} and solving for $D'$ now yields
\eqref{eq:corrected-ode}. For $a<x<b$,
\[
 1-2p(s)=\frac{E(\eta)}{2k-E(\eta)}>0,\qquad
 U-D=2R+C>0.
\]
For $2<k<3$, the latter denominator equals $4(k-2)>0$ at $x=a$.
Thus the right-hand side of \eqref{eq:corrected-ode} is continuous in $x$
and locally Lipschitz in $D$ throughout the stated region, so the
initial-value problem has at most one solution.  At $k=2$ the same computations are valid
away from the singular endpoint, and the endpoint values established above
give the stated continuous extension.
\end{proof}

\begin{proposition}[Feasibility of the explicit coupling]\label{prop:feasible-plan}
Let $D_k,U_k$ be the maps in \cref{prop:matching}, and, for $x>-1$, put
\begin{equation*}
 q_k(x):=\frac{x-D_k(x)}{U_k(x)-D_k(x)}.
\end{equation*}
Set $q_k(-1)=1/2$; this arbitrary choice is immaterial because
$D_k(-1)=U_k(-1)=-1$ and $\mu$ has no atom there.
The maps are continuous, hence Borel measurable, and
$q_k(x)\in[0,1]$ by \cref{prop:matching}.
Then
\begin{equation}\label{eq:pi-k}
 \pi_k(\dd x,\dd y)
 =\mu(\dd x)\left[(1-q_k(x))\delta_{D_k(x)}(\dd y)
                  +q_k(x)\delta_{U_k(x)}(\dd y)\right]
\end{equation}
belongs to $\M_k(\mu,\nu)$.
\end{proposition}

\begin{proof}
The definition of $q_k$ gives the martingale property pointwise. The
constraint follows from $D_k(x)\geq x-k$, with equality on $[b,1]$.
To check the target marginal, for $a<x<b$ put
$C(\eta):=2k-2-\eta$. Equations
\eqref{eq:x-middle}--\eqref{eq:U-middle} give
\[
 q=\frac{R+C}{2R+C},\qquad
 U'(x)=\frac{2(R+C)}{2R+C}=2q,\qquad
 D'(x)=-\frac{2(2k-E)R}{E(2R+C)}.
\]
Thus transport along $U|_{[a,b]}$ has target density $1/4$.
For $y=D_m(\eta)\in(-2,1-k)$, there are two lower-map preimages: $x_m(\eta)$
and $s=y+k=\rho(\eta)\in[b,1]$. Their target densities sum to
\[
 \frac{(1-q)/2}{-D'}
 +\frac{p(s)}2
 =\frac{E}{4(2k-E)}
  +\frac{k-E}{2(2k-E)}
 =\frac14.
\]
The maps on $[-1,a]$ each give target density $1/4$, as does
$U|_{[b,1]}$ by \eqref{eq:tail-ode}. The images are
\begin{center}
\begin{tabular}{@{}lll@{}}
\toprule
Source interval & Image under $D$ & Image under $U$\\
\midrule
$[-1,a]$ & $[1-k,-1]$ & $[-1,3k-7]$\\
$[a,b]$ & $[-2,1-k]$ & $[3k-7,\eta_b]$\\
$[b,1]$ & $[-2,1-k]$ & $[\eta_b,2]$\\
\bottomrule
\end{tabular}
\end{center}
The lower-map contributions sum to $1/4$ on $[-2,1-k]$; the remaining
images partition $[1-k,2]$ up to endpoints, each with density $1/4$.
Degenerate source intervals at $k=2,3$ carry no mass. The second
marginal is therefore $\nu$.
\end{proof}

\subsection{The regime
\texorpdfstring{$1<k<2$}{1 < k < 2}}\label{sec:small-k}

We complete the construction for $1<k<2$ by solving the residual
marginal equations on $[-1,b]$. Since $1-k>-1$, the formulas change
at a point $x_0\in(-1,b)$ such that
\[
 \begin{array}{ll}
 -2\leq D(x)\leq U(x)\leq1-k,& -1\leq x\leq x_0,\\
 -2\leq D(x)\leq1-k\leq U(x)\leq\eta_b,& x_0\leq x\leq b,
 \end{array}
 \qquad U(x_0)=1-k.
\]
We determine $x_0$ and the maps on $[x_0,b]$ explicitly in a parameter.
On $[-1,x_0]$, a change of variable satisfying
$\bar r(y(\delta))y'(\delta)=1/4$ reduces the two marginal equations to
one strictly monotone scalar equation. The source and target intervals
are shown in \cref{fig:construction-geometry}.

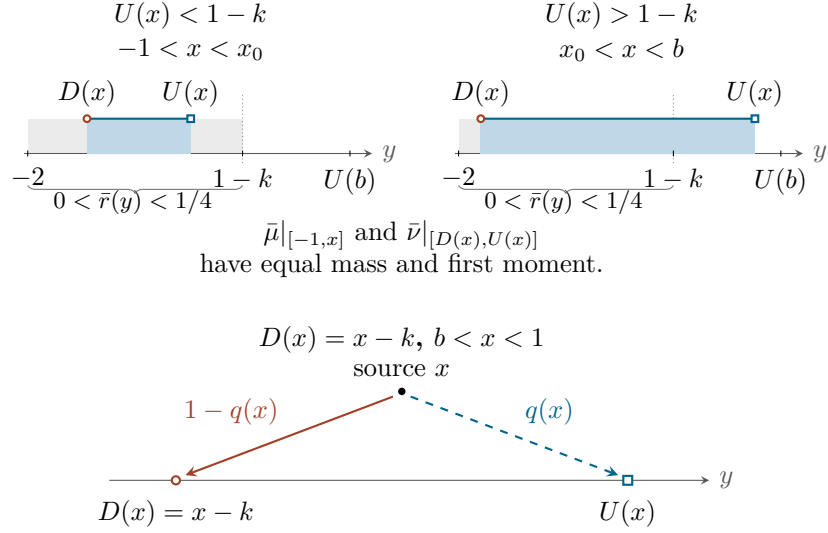
\begin{figure}[tbp]
\centering
\begin{tikzpicture}[x=0.98cm,y=0.98cm,font=\small,>=stealth]
\foreach \shift/\title/\range/\leftend/\rightend in {
  0/{$U(x)<1-k$}/{$-1<x<x_0$}/0.8/2.2,
  5.7/{$U(x)>1-k$}/{$x_0<x<b$}/0.3/4.0} {
  \begin{scope}[xshift=\shift cm]
    \node[font=\small\bfseries] at (2.2,2.55) {\title};
    \node at (2.2,2.08) {\range};
    \fill[black!8] (0,0.68) rectangle (2.9,1.15);
    \draw[black!45,densely dotted] (2.9,0.55)--(2.9,1.48);
    \draw[->,black!65] (-0.05,0.68)--(4.65,0.68) node[right] {$y$};
    \fill[MidnightBlue!18] (\leftend,0.68) rectangle (\rightend,1.15);
    \draw[MidnightBlue!85!black,thick] (\leftend,1.15)--(\rightend,1.15);
    \draw[BrickRed!85!black,thick,fill=white] (\leftend,1.15) circle (0.045);
    \draw[MidnightBlue!90!black,thick,fill=white]
      (\rightend-0.045,1.105) rectangle (\rightend+0.045,1.195);
    \node[above] at (\leftend,1.18) {$D(x)$};
    \node[above] at (\rightend,1.18) {$U(x)$};
    \foreach \position/\label in {0/{-2},2.9/{1-k},4.35/{U(b)}} {
      \draw (\position,0.64)--(\position,0.72);
      \node[below] at (\position,0.62) {$\label$};
    }
    \node[font=\footnotesize] at (1.42,0.03) {$0<\bar r(y)<1/4$};
    \draw[decorate,decoration={brace,mirror,amplitude=3pt},black!55]
      (0,0.25)--(2.9,0.25);
  \end{scope}
}
\node[align=center,font=\small] at (5.05,-0.6)
  {$\bar\mu|_{[-1,x]}$ and $\bar\nu|_{[D(x),U(x)]}$\\
   have equal mass and first moment.};
\begin{scope}[yshift=-4.0cm]
  \node[font=\small\bfseries] at (5.05,2.25) {$D(x)=x-k$, $b<x<1$};
  \draw[->,black!65] (1.1,0.35)--(9.2,0.35) node[right] {$y$};
  \draw[BrickRed!85!black,thick,fill=white] (2.0,0.35) circle (0.055);
  \draw[MidnightBlue!90!black,thick,fill=white] (8.04,0.29) rectangle (8.16,0.41);
  \node[below] at (2.0,0.22) {$D(x)=x-k$};
  \node[below] at (8.1,0.22) {$U(x)$};
  \fill (5.05,1.55) circle (0.05);
  \node[above] at (5.05,1.62) {source $x$};
  \draw[->,BrickRed!85!black,thick] (4.93,1.48)--(2.13,0.42)
     node[midway,above left] {$1-q(x)$};
  \draw[->,MidnightBlue!90!black,thick,dashed] (5.17,1.48)--(7.97,0.42)
     node[midway,above right] {$q(x)$};
\end{scope}
\end{tikzpicture}
\caption{The construction for $1<k<2$ (schematic). Upper panels:
for each $x\in[-1,b]$, the restriction of the coupling to
$[-1,x]\times\mathbb R$ has second marginal
$\bar\nu|_{[D(x),U(x)]}$. The formulas change at $U(x_0)=1-k$.
Lower panel: for $x\in[b,1]$, the conditional law is
$(1-q(x))\delta_{x-k}+q(x)\delta_{U(x)}$, with mean $x$.}
\label{fig:construction-geometry}
\end{figure}

For $0\leq\delta\leq1$, define
\begin{equation*}
 z(\delta):=2-\delta+2k\log\frac{k-1}{k-\delta},
 \qquad y(\delta):=z(\delta)-k.
\end{equation*}
Substituting $\delta=U-x$ into \eqref{eq:tail-implicit} gives
$x=z(\delta)$. Moreover, $z(1)=1$ and
\[
 z'(\delta)=\frac{k+\delta}{k-\delta}\geq1,
 \qquad
 z(0)-b=4-k+2k\log\left(1-\frac1k\right)
 <-\frac{(k-1)^2}{k}<0.
\]
The logarithmic bound follows from $\log(1-u)<-u-u^2/2$ with
$u=1/k$. Since $z$ is strictly increasing, there is a unique
$\delta_b\in(0,1)$ satisfying $z(\delta_b)=b$, and
$z:[\delta_b,1]\to[b,1]$ is a $C^1$ bijection. In particular,
$y(\delta_b)=-2$, $y(1)=1-k$, and $U(b)=b+\delta_b=\eta_b$.
The maps on $[b,1]$ become
\begin{equation}\label{eq:tail-delta}
  x=z(\delta),\qquad D=x-k,\qquad U=x+\delta,
  \qquad \delta_b\leq\delta\leq1.
\end{equation}
For a source $x=z(\delta)$, the lower weight is
$p(x)=\delta/(k+\delta)$. Hence
\begin{equation}\label{eq:uniform-mass-coordinate}
 z'(\delta)=\frac{k+\delta}{k-\delta},\qquad
 \bar r(y(\delta))=\frac{k-\delta}{4(k+\delta)},\qquad
 \bar r(y(\delta))y'(\delta)=\frac14.
\end{equation}
Define
\begin{equation}\label{eq:W-delta}
\begin{split}
 \mathcal W(\delta):={}&b^2-2b(\delta-1)\\
 &+2k\left[(k-\delta)\log\frac{k-\delta}{k-1}
                         +\delta-1\right].
\end{split}
\end{equation}

An antiderivative of $y$ is
\begin{equation}\label{eq:A-delta}
\begin{split}
 A(\delta)={}&(2-k)\delta-\frac{\delta^2}{2}\\
 &+2k\left[\delta\log(k-1)
 +(k-\delta)\log(k-\delta)-(k-\delta)\right].
\end{split}
\end{equation}

\begin{lemma}[Marginal equations when $D\leq1-k\leq U$]
\label{lem:small-middle-reduction}
Fix $\delta\in[\delta_b,1]$ and set $D=y(\delta)$. Suppose
$-1\leq x\leq b$, $D\leq x\leq U$, and $1-k\leq U\leq\eta_b$.
Then
\eqref{eq:residual-mass}--\eqref{eq:residual-moment} are equivalent to
\begin{equation}\label{eq:small-middle-reduction}
 x=b-\delta+W,\qquad U=b-\delta+2W,
 \qquad W^2=\mathcal W(\delta),\qquad W\geq0.
\end{equation}
\end{lemma}

\begin{proof}
By \eqref{eq:uniform-mass-coordinate}, splitting the target integrals
at $y(1)=1-k$ gives
\begin{align*}
 \int_D^U\bar r(y)\dd y
 &=\frac{1-\delta}{4}+\frac{U-(1-k)}4,\\
 \int_D^Uy\bar r(y)\dd y
 &=\frac{A(1)-A(\delta)}4
   +\frac{U^2-(1-k)^2}{8}.
\end{align*}
The mass equation is $U=2x-b+\delta$. With $W=U-x$, this gives
the formulas for $x,U$ in \eqref{eq:small-middle-reduction}. Substituting
them and \eqref{eq:A-delta} into the first-moment equation gives
$W^2=\mathcal W(\delta)$. Each substitution is reversible, proving
the equivalence.
\end{proof}

Define
\begin{equation}\label{eq:T-delta}
 T(\delta):=A(1)-A(\delta)
 -\left(\frac{(1+\delta)^2}{4}-1\right).
\end{equation}

\begin{lemma}[The root of $T$ and positivity of $\mathcal W$]\label{lem:small-k-transition}
There is a unique $\delta_0\in(\delta_b,1)$ such that
$T(\delta_0)=0$, with $T<0$ on $[\delta_b,\delta_0)$ and
$T>0$ on $(\delta_0,1)$. Furthermore,
\[
 \mathcal W(\delta)>0\quad(\delta_b\leq\delta\leq\delta_0),
 \qquad
 \sqrt{\mathcal W(\delta_b)}=\delta_b,
 \qquad
 \sqrt{\mathcal W(\delta_0)}=\frac{3-2k+\delta_0}{2}>0.
\]
\end{lemma}

\begin{proof}
Since $A'=y$, one has
\[
 T''(\delta)=-y'(\delta)-\frac12<0,\qquad
 T(1)=0,\qquad T'(1)=k-2<0.
\]
Therefore $T>0$ immediately to the left of $1$. Strict convexity
of $y$, together with $y(\delta_b)=-2$ and $y(1)=1-k$, gives
\[
 \frac{A(1)-A(\delta_b)}{1-\delta_b}
 <\frac{y(\delta_b)+y(1)}2
 =-\frac{1+k}{2}< -\frac{3+\delta_b}{4}.
\]
The last inequality uses $\delta_b<1<2k-1$. Since
\[
 \frac{(1+\delta_b)^2/4-1}{1-\delta_b}
 =-\frac{3+\delta_b}{4},
\]
we obtain $T(\delta_b)<0$. Strict concavity gives exactly one zero
$\delta_0$ in $(\delta_b,1)$ and the asserted signs of $T$.

Put
\[
 \alpha(\delta):=\frac{3-2k+\delta}{2}.
\]
Equations \eqref{eq:W-delta}--\eqref{eq:T-delta} give
\begin{equation}\label{eq:W-T-identity}
 \mathcal W(\delta)-\alpha(\delta)^2=-T(\delta).
\end{equation}
To prove $\alpha(\delta_0)>0$, set
$t=2k-3$. If $t\leq\delta_b$, then $t<\delta_0$. Otherwise
$t\in(\delta_b,1)$; writing $\varepsilon=2-k\in(0,1/2)$ gives
\begin{align*}
 \mathcal W(t)
 &=k^2-4(3-k)
   +2k(3-k)\log\frac{3-k}{k-1}\\
 &>k^2-4(3-k)+4k(3-k)\varepsilon
   =\varepsilon^2(5-4\varepsilon)>0,
\end{align*}
where
$\log((1+\varepsilon)/(1-\varepsilon))>2\varepsilon$.
As $\alpha(t)=0$, \eqref{eq:W-T-identity} yields $T(t)<0$,
so again $t<\delta_0$. Hence $\alpha(\delta_0)>0$.

For $\delta<\delta_0$, \eqref{eq:W-T-identity} and $T(\delta)<0$
give $\mathcal W(\delta)>\alpha(\delta)^2$; at $\delta_0$ they give
$\mathcal W(\delta_0)=\alpha(\delta_0)^2>0$. Finally,
$z(\delta_b)=b$ is equivalent to
\[
 2k\log\frac{k-\delta_b}{k-1}=2-b-\delta_b.
\]
Substitution into \eqref{eq:W-delta} gives
$\mathcal W(\delta_b)=\delta_b^2$, completing the proof.
\end{proof}

Set $W(\delta)=\sqrt{\mathcal W(\delta)}$ for
$\delta\in[\delta_b,\delta_0]$ and define
\begin{equation}\label{eq:middle-delta}
 x=b-\delta+W(\delta),\qquad
 D=z(\delta)-k,
 \qquad U=b-\delta+2W(\delta).
\end{equation}
The endpoint values are
\begin{align*}
 (x,D,U)(\delta_b)&=(b,-2,b+\delta_b),\\
 (x,D,U)(\delta_0)&=
 \left(-\frac{1+\delta_0}{2},y(\delta_0),1-k\right).
\end{align*}
The second identity uses $W(\delta_0)=\alpha(\delta_0)$. Set
$x_0=-(1+\delta_0)/2$. Moreover,
\[
 U(\delta)-(1-k)=2\bigl(W(\delta)-\alpha(\delta)\bigr)>0,
 \qquad \delta_b\leq\delta<\delta_0.
\]

On $[-1,x_0]$, write $D=y(L)$ and $U=y(R)$. There is a unique
$\delta_*\in(\delta_b,1)$ with $y(\delta_*)=-1$, since $y$ increases
strictly from $-2$ to $1-k>-1$. Set $L(-1)=R(-1)=\delta_*$.
By \eqref{eq:uniform-mass-coordinate}, the residual mass and
first-moment equations become
\begin{align}
 R(x)-L(x)&=2(x+1),\label{eq:LR-mass}\\
 A(R(x))-A(L(x))&=x^2-1,\label{eq:LR-moment}
\end{align}
for $-1<x\leq x_0$.

\begin{proposition}[Construction for $1<k<2$]\label{prop:small-k}
For every $x\in(-1,x_0]$, equations
\eqref{eq:LR-mass}--\eqref{eq:LR-moment} have a unique solution satisfying
\[
 \delta_b\leq L(x)<R(x)\leq1.
\]
At $x=x_0$ this solution is $(L,R)=(\delta_0,1)$.  As $x\downarrow-1$
it converges to the prescribed endpoint $L=R=\delta_*$.
Together with \eqref{eq:middle-delta} on $[x_0,b]$ and
\eqref{eq:tail-delta} on $[b,1]$, these maps and their martingale
weights define $\pi_k\in\M_k(\mu,\nu)$. The maps are unique up to
$\mu$-null sets among constructions with the prescribed transport
from $[b,1]$, decreasing $D$ and increasing $U$ on $[-1,b]$, and
\eqref{eq:residual-mass}--\eqref{eq:residual-moment}.
Optimality over $\M_k(\mu,\nu)$ is proved in
\cref{thm:small-optimality}.
\end{proposition}

\begin{proof}
\emph{Existence on $[-1,x_0]$.}
Put $\Delta=2(x+1)$ and $R=L+\Delta$. The first-moment equation is
$G_\Delta(L)=0$, where
\[
 G_\Delta(L):=A(L+\Delta)-A(L)
 -\left(\frac{\Delta^2}{4}-\Delta\right),
 \qquad \delta_b\leq L\leq1-\Delta.
\]
Its derivative is
\[
 \partial_LG_\Delta(L)=y(L+\Delta)-y(L)>0,
\]
so there is at most one root. Let $\Delta_0=1-\delta_0$. At the lower
endpoint, $g_b(\Delta):=G_\Delta(\delta_b)$ satisfies
\[
 g_b(0)=0,\qquad
 g_b''(\Delta)=y'(\delta_b+\Delta)-\frac12>0,
 \qquad g_b(\Delta_0)<G_{\Delta_0}(\delta_0)=0.
\]
Convexity therefore gives $g_b(\Delta)<0$ for
$0<\Delta\leq\Delta_0$. At the upper endpoint,
\[
 g_1(\Delta):=G_\Delta(1-\Delta),\qquad
 g_1(0)=g_1(\Delta_0)=0,\qquad
 g_1''(\Delta)=-y'(1-\Delta)-\frac12<0.
\]
Thus $g_1(\Delta)>0$ for $0<\Delta<\Delta_0$. The intermediate value
theorem proves existence, and at $x=x_0$ the root is
$(L,R)=(\delta_0,1)$.

As $x\downarrow-1$, division by $\Delta$ gives
\[
 \frac1\Delta\int_L^{L+\Delta}y(s)\dd s
 =\frac\Delta4-1.
\]
Since $R-L=\Delta\to0$ and $y$ is uniformly continuous on
$[\delta_b,1]$, the left-hand side differs from $y(L)$ by a quantity
tending to zero. Thus $y(L)\to-1$. The continuous inverse of $y$ gives
$L\to\delta_*$ and then $R\to\delta_*$.

\medskip\noindent\emph{Ordering and monotonicity on $[-1,x_0]$.}
Divide
\eqref{eq:LR-moment} by $\Delta$ to obtain
\[
 m:=\frac1\Delta\int_L^R y(s)\dd s=\frac{x-1}{2}.
\]
Since $y$ is strictly increasing, $D<m<U$, and $m<x$ gives $D<x$.
Also $y$ is convex and $y'>1$, so the Hermite--Hadamard inequality
gives
\[
 U-m\geq\frac{U-D}{2}>\frac{R-L}{2}=\frac\Delta2.
\]
But $x-m=\Delta/4$, hence $U>x$. For $x\in(-1,x_0)$, the root lies in
$(\delta_b,1-\Delta)$ and $\partial_LG_\Delta(L)>0$. The implicit
function theorem therefore permits differentiation, which gives
\begin{equation}\label{eq:LR-derivatives}
 L'=-2\frac{U-x}{U-D}<0,
 \qquad
 R'=2\frac{x-D}{U-D}>0.
\end{equation}
Hence $D$ decreases and $U$ increases. At $x=x_0$, their values
$D=y(\delta_0)$ and $U=1-k$ agree with \eqref{eq:middle-delta}.

\medskip\noindent\emph{Ordering and monotonicity on $[x_0,b]$.}
Put
\[
 \eta(\delta):=z(\delta)+\delta,
 \qquad C(\delta):=2k-2-\eta(\delta).
\]
Direct differentiation and \eqref{eq:middle-delta} give
\begin{gather}
 \mathcal W'=-C,\qquad W'=-\frac{C}{2W},\notag\\
 x-D=W+C,\qquad U-D=2W+C.\label{eq:middle-gaps}
\end{gather}
Let $J=W+C=x-D$. At $\delta_b$, $J=k>0$; at $\delta_0$,
the ordering at $x_0$ gives $J=x_0-y(\delta_0)>0$.
At any zero of $J$ one would have $C=-W$ and
\[
 J'=\frac12-\frac{2k}{k-\delta}<0.
\]
Every zero would have negative derivative, which is incompatible
with positivity at both endpoints. Thus
$W+C>0$ and $2W+C>0$ throughout. In particular,
\begin{equation*}
 x'=-\frac{2W+C}{2W}<0,\qquad
 D_\delta=\frac{k+\delta}{k-\delta}>0,\qquad
 U_\delta=-\frac{W+C}{W}<0.
\end{equation*}
Thus $\delta\mapsto x(\delta)$ is a strictly decreasing bijection from
$[\delta_b,\delta_0]$ onto $[x_0,b]$; in particular, $x_0<b$.
Equations
\eqref{eq:middle-gaps} and $U-x=W>0$ show that $D<x<U$.
Consequently $D$ decreases and $U$ increases as functions of $x$, with
\begin{equation*}
 q=\frac{x-D}{U-D}=\frac{W+C}{2W+C}\in(0,1),
 \qquad U_x=\frac{2(W+C)}{2W+C}=2q.
\end{equation*}

\medskip\noindent\emph{The constraint $D(x)\geq x-k$.}
On $[x_0,b]$, let $g=D-x+k$.
Then $g(\delta_b)=0$ and
\[
 g_\delta=\frac{k+\delta}{k-\delta}
          +\frac{2W+C}{2W}>0.
\]
Thus $D>x-k$ on $[x_0,b)$. On $[-1,x_0]$,
$g_x=D_x-1<0$ by \eqref{eq:LR-derivatives}, so $g\geq g(x_0)>0$.
On $[b,1]$, $D=x-k$ by definition.

\medskip\noindent\emph{The target marginal.}
For $x\in(x_0,b)$,
\[
 D_x=-\frac{2W(k+\delta)}{(k-\delta)(2W+C)},
 \qquad
 \frac{(1-q)/2}{-D_x}
   =\frac{k-\delta}{4(k+\delta)}=\bar r(D).
\]
Also $U>1-k$ and $U_x=2q$, so the upper-map target density is
$(q/2)/U_x=1/4=\bar r(U)$. For $x\in(-1,x_0)$,
\eqref{eq:LR-derivatives} gives $R'=2q$ and $-L'=2(1-q)$, hence
\[
 \frac{q/2}{U_x}=\frac1{4y'(R)}=\bar r(U),
 \qquad
 \frac{(1-q)/2}{-D_x}=\frac1{4y'(L)}=\bar r(D).
\]
The images partition $[-2,b+\delta_b]$ up to endpoints:
\begin{center}
\begin{tabular}{@{}ll@{}}
\toprule
Map and source interval & Target image\\
\midrule
$D|_{[-1,x_0]}$ & $[y(\delta_0),-1]$\\
$U|_{[-1,x_0]}$ & $[-1,1-k]$\\
$D|_{[x_0,b]}$ & $[-2,y(\delta_0)]$\\
$U|_{[x_0,b]}$ & $[1-k,b+\delta_b]$\\
\bottomrule
\end{tabular}
\end{center}
The map $U|_{[b,1]}$ gives density $1/4$ on $[b+\delta_b,2]$.
For $x=z(\delta)\in[b,1]$, transport to $D(x)=x-k=y(\delta)$
contributes $\delta/[2(k+\delta)]$; its sum with
$\bar r(y(\delta))$ is $1/4$.

The maps agree at $x_0,b$ and are continuous. Set
\[
 q(x):=\frac{x-D(x)}{U(x)-D(x)}\quad(x>-1),
 \qquad q(-1):=\frac12.
\]
Then
\begin{equation*}
 \pi_k(\dd x,\dd y)
 =\frac12\1_{[-1,1]}(x)\dd x
  \bigl[(1-q(x))\delta_{D(x)}(\dd y)
       +q(x)\delta_{U(x)}(\dd y)\bigr]
\end{equation*}
has conditional mean $x$, is supported on $y\geq x-k$, and has second
marginal $\nu$. Thus $\pi_k\in\M_k(\mu,\nu)$. Uniqueness within
the stated class follows from the interval-moment argument after
\eqref{eq:residual-moment}, together with uniqueness for
\eqref{eq:tail-ode}.
\end{proof}

\begin{remark}[The upper displacement]
For $k$ sufficiently close to $1$, $U(x)-x$ decreases on a subinterval
of $(x_0,b)$, although $U$ increases on $[-1,1]$. The proof in
\cref{sec:small-optimality} therefore uses monotonicity of $U$ without
assuming monotonicity of $U-x$.
\end{remark}

\section{Optimality}\label{sec:optimality-limitations}

For $x\in[-1,1]$, define the admissible conditional laws by
\[
 \mathcal G_k(x):=\left\{m\in\Pp(I_k(x)):
                 \int y\,m(\dd y)=x\right\},
 \qquad I_k(x)=[-2,2]\cap[x-k,\infty).
\]
For $\psi\in C([-2,2])$, set
\[
 R_{k,h}\psi(x):=\inf_{m\in\mathcal G_k(x)}
                  \int [h(y-x)-\psi(y)]\,m(\dd y).
\]
The constrained duality theorem
\cite[Theorem~5.1 and Corollary~5.1]{BayraktarZhangZhou} gives
\[
 V_k(h)=\sup_{\psi\in C([-2,2])}
          \left\{\int R_{k,h}\psi(x)\mu(\dd x)
                        +\int\psi(y)\nu(\dd y)\right\}.
\]
The state intervals are compact, the martingale and support constraints
are closed and convex, and the cost is continuous. Adding a constant to
$h$ makes it nonnegative, as required by that theorem, without changing
the minimizers. The theorem does not assert dual attainment.

For smooth costs, we prove optimality by constructing a continuous
$\psi$ for which
\begin{equation}\label{eq:conditional-optimality}
 \pi_{k,x}\in\operatorname*{arg\,min}_{m\in\mathcal G_k(x)}
           \int [h(y-x)-\psi(y)]\,m(\dd y)
 \qquad\text{for }\mu\text{-almost every }x,
\end{equation}
where $\pi_{k,x}=(1-q(x))\delta_{D(x)}+q(x)\delta_{U(x)}$.
More precisely, we find bounded functions $\phi,\theta$ such that
\[
 G_x(y):=h(y-x)-\phi(x)-\psi(y)-\theta(x)(y-x)\geq0,
 \qquad x\in[-1,1],\quad y\in I_k(x),
\]
with equality at $D(x)$ and $U(x)$. For any $m\in\mathcal G_k(x)$,
integration eliminates the term $\theta(x)(y-x)$ and gives
$\int[h(y-x)-\psi(y)]\,m(\dd y)\geq\phi(x)$.
The law $\pi_{k,x}$ attains equality, so
$R_{k,h}\psi(x)=\phi(x)$ and \eqref{eq:conditional-optimality} holds.
Integrating against the source marginal then yields, for every
$\pi\in\M_k(\mu,\nu)$,
\[
 \int h(y-x)\dd\pi\geq\mu(\phi)+\nu(\psi)
 =\int h(y-x)\dd\pi_k.
\]
Uniform approximation extends the result to $C^1$ costs with convex
derivative. We first treat $2\leq k<3$; the additional estimates for
$1<k<2$ are given in \cref{sec:small-optimality}.

\subsection{Optimality for
\texorpdfstring{$2\leq k<3$}{2 <= k < 3}}
\label{sec:duality}

We prove optimality for $2\leq k<3$ by constructing dual functions
such that $G_x\geq0$ on $I_k(x)$, with equality at $D(x)$ and $U(x)$,
for every $x\in[-1,1]$.

Fix $2\leq k<3$ and set
\begin{equation*}
 d(x):=D(x)-x,\qquad w(x):=U(x)-x,\qquad B(x):=x-k.
\end{equation*}
By \cref{prop:matching}, $D:[-1,b]\to[-2,-1]$ is strictly decreasing,
$U:[-1,1]\to[-1,2]$ is strictly increasing,
$-k\leq d\leq0$, and $w\geq0$ is nondecreasing.
For $x\in[b,1]$, the constraint boundary $B(x)$ belongs to
$D([-1,b])$. Define
\begin{equation*}
 \alpha(x):=D^{-1}(B(x)),\qquad
 \zeta(x):=B(x)-\alpha(x)=d(\alpha(x)),
\end{equation*}
where $D^{-1}$ always denotes the inverse of $D|_{[-1,b]}$.
Thus $\alpha(x)\leq b\leq x$, $\alpha(b)=b$,
$\alpha(1)=2k-5$, $-k\leq\zeta\leq0$, and
\begin{equation}\label{eq:x-a}
 x-\alpha(x)=\zeta(x)+k.
\end{equation}

\medskip\noindent\emph{Auxiliary probability measures.}

We construct probability measures $Q_x$ on the displacement variable.
They will determine the coefficient
$S_h(x)=\int h''\dd Q_x$ in the dual functions.
For $-1<x\leq b$, set
\begin{equation}\label{eq:Q-free}
 Q_x:=\Unif[d(x),w(x)],
\end{equation}
and let $Q_{-1}:=\delta_0$. For $b\leq x\leq1$, define $Q_x$ by
\begin{equation}\label{eq:Q-active}
 Q_x:=
 \frac{\lambda_{[\zeta(x),w(x)]}
       +\displaystyle\int_{\alpha(x)}^xQ_s\dd s}
      {w(x)+k}.
\end{equation}
Here $\lambda_{[r,s]}$ is Lebesgue measure restricted to $[r,s]$,
and the integral of measures is defined weakly. At $x=b$, this formula
agrees with \eqref{eq:Q-free}.

To prove existence and uniqueness, let $\mathcal X$ be the Banach space
of bounded weakly measurable kernels of signed measures on $[-k,1]$,
indexed by $x\in[b,1]$, with the supremum total-variation norm. Using the known measures on $[-1,b]$,
define
\[
 (\mathcal TK)_x:=\frac{1}{w(x)+k}\left(
 \lambda_{[\zeta(x),w(x)]}
 +\int_{\alpha(x)}^b Q_s\dd s+\int_b^xK_s\dd s\right).
\]
The endpoint maps are continuous, so $\mathcal T$ preserves
measurability. Since $w+k\geq k$,
\[
 \|(\mathcal TQ)_x-(\mathcal TR)_x\|_{\mathrm{TV}}
 \leq\frac1k\int_b^x\|Q_s-R_s\|_{\mathrm{TV}}\dd s,
\]
and iteration gives
\[
 \|\mathcal T^nQ-\mathcal T^nR\|_{\mathcal X}
 \leq\frac{(1-b)^n}{k^n n!}\|Q-R\|_{\mathcal X}.
\]
Picard iteration from zero therefore converges to a unique fixed point,
which is nonnegative. Its total mass satisfies the scalar Volterra
equation obtained from \eqref{eq:Q-active}. By \eqref{eq:x-a},
\[
 (w(x)-\zeta(x))+(x-\alpha(x))=w(x)+k,
\]
so the constant function $1$ solves this equation. Uniqueness proves
$Q_x(\mathbb R)=1$. Moreover,
\begin{equation}\label{eq:Q-support}
 \supp Q_x\subset[-k,w(x)].
\end{equation}
Indeed, this holds for the measures in \eqref{eq:Q-free}; in
\eqref{eq:Q-active}, $\zeta(x)\geq-k$ and $w(s)\leq w(x)$ for
$s\leq x$. Thus every Picard iterate has the asserted support, and so
does its limit.

\begin{lemma}[A stochastic lower bound]\label{lem:stochastic-order}
For every $x\in[-1,1]$,
\begin{equation}\label{eq:stochastic-order}
 \Unif[-k,w(x)]\leq_{\mathrm{st}}Q_x.
\end{equation}
Here $R\leq_{\mathrm{st}}Q$ means
$\int f\dd R\leq\int f\dd Q$ for every bounded increasing Borel
function $f$.
\end{lemma}

\begin{proof}
For $x\leq b$, the assertion follows from $d(x)\geq-k$, with equality
of the measures at $b$. Fix a bounded increasing $f$ and write
\[
 q_f(x):=\int f\dd Q_x,\qquad
 u_f(t):=\frac1{t+k}\int_{-k}^t f(r)\dd r,\quad t>-k.
\]
The function $u_f$ is increasing. If $x>b$, then $\zeta(x)>-k$,
because $D(s)>s-k$ for $s<b$. With $w=w(x)$ and
$\zeta=\zeta(x)$, \eqref{eq:Q-active} and \eqref{eq:x-a} give
\begin{align*}
 (w+k)\bigl(q_f(x)-u_f(w)\bigr)
 &=\int_{\alpha(x)}^xq_f(s)\dd s-\int_{-k}^{\zeta}f(r)\dd r\\
 &=\int_{\alpha(x)}^x\bigl(q_f(s)-u_f(\zeta)\bigr)\dd s.
\end{align*}
Since $w(s)\geq0\geq\zeta$, we have
$u_f(w(s))\geq u_f(\zeta)$. Put
$e(s):=[u_f(w(s))-q_f(s)]_+$. The integrand is nonnegative for
$s\leq b$ and is at least $-e(s)$ for $s\geq b$. Hence
\[
 e(x)\leq\frac1k\int_b^xe(s)\dd s.
\]
Gronwall's lemma gives $e=0$, proving \eqref{eq:stochastic-order}.
\end{proof}

\medskip\noindent\emph{Dual functions and equality at the support points.}

For $h\in C^3([-3,3])$ with $h'''\geq0$, set
\begin{equation}\label{eq:S-theta}
 S_h(x):=\int h''(r)Q_x(\dd r),\qquad
 \theta(x):=-\int_{-1}^xS_h(s)\dd s.
\end{equation}
Since $S_h$ is bounded and measurable, $\theta$ is Lipschitz.
Define $\psi(-1)=0$ and
\begin{align}
 \psi'(y)
 &=h'\bigl(y-D^{-1}(y)\bigr)-\theta(D^{-1}(y)),
 &&-2\leq y\leq-1,\label{eq:psi-lower}\\
 \psi'(y)
 &=h'\bigl(y-U^{-1}(y)\bigr)-\theta(U^{-1}(y)),
 &&-1\leq y\leq2.\label{eq:psi-upper}
\end{align}
The inverse maps are continuous, and both formulas equal $h'(0)$ at
$y=-1$. Thus $\psi\in C^1([-2,2])$. Set
\begin{equation}\label{eq:phi-def}
 \phi(x):=h(w(x))-\psi(U(x))-\theta(x)w(x),
\end{equation}
and write
\begin{equation*}
 J(x,y):=h(y-x)-\psi(y)-\theta(x)(y-x),
 \qquad G_x(y)=J(x,y)-\phi(x).
\end{equation*}
All three functions $\phi,\psi,\theta$ are continuous on compact
intervals and hence bounded.

\begin{lemma}[Equality at the support points]\label{lem:contacts}
The functions above satisfy
\begin{align}
 J(x,D(x))&=J(x,U(x))=\phi(x),&&-1\leq x\leq b,
 \label{eq:contact-free}\\
 J(x,B(x))&=J(x,U(x))=\phi(x),&&b\leq x\leq1.
 \notag
\end{align}
\end{lemma}

\begin{proof}
On each open interval where the transport maps are smooth, the contact
differences below are locally absolutely continuous. We show that their
derivatives vanish almost everywhere, so each difference is constant
on that interval. Continuity matches these constants across the
transition points and at $x=-1$.
For $-1<x<b$, \eqref{eq:Q-free} gives
\begin{equation*}
 S_h(x)=\frac{h'(w(x))-h'(d(x))}{w(x)-d(x)}.
\end{equation*}
Equations \eqref{eq:psi-lower}--\eqref{eq:psi-upper} imply
$J_y(x,D(x))=J_y(x,U(x))=0$. Consequently,
\[
 \frac{\dd}{\dd x}[J(x,U(x))-J(x,D(x))]
 =-[h'(w)-h'(d)]+S_h(w-d)=0.
\]
The difference is zero at $x=-1$, proving \eqref{eq:contact-free}.

For $b<x<1$, \eqref{eq:Q-active} and \eqref{eq:S-theta} give
\begin{align}
 (w+k)S_h
 &=h'(w)-h'(\zeta)+\int_{\alpha(x)}^xS_h(s)\dd s\notag\\
 &=h'(w)-h'(\zeta)+\theta(\alpha(x))-\theta(x).
 \label{eq:S-active}
\end{align}
Also $\psi'(B(x))=h'(\zeta(x))-\theta(\alpha(x))$. Hence
\[
 \frac{\dd}{\dd x}[J(x,U(x))-J(x,B(x))]
 =-h'(w)+\theta(x)+(w+k)S_h+\psi'(B(x))=0.
\]
At $x=b$, $B(b)=D(b)$, so \eqref{eq:contact-free} proves the equality
on $[b,1]$.
\end{proof}

Since $h''$ is increasing, \cref{lem:stochastic-order} and
\eqref{eq:Q-support} imply
\begin{equation*}
 \frac{h'(w)-h'(-k)}{w+k}\leq S_h(x)\leq h''(w).
\end{equation*}
In particular, for $b\leq x\leq1$, \eqref{eq:S-active} gives
\begin{align}
 J_y(x,B(x)+)
 &=h'(-k)-\psi'(B(x))-\theta(x)\notag\\
 &=(w+k)S_h-\bigl(h'(w)-h'(-k)\bigr)\geq0.
 \label{eq:boundary-sign}
\end{align}

\medskip\noindent\emph{The dual inequality for truncated quadratic costs.}

\begin{lemma}\label{lem:hinge}
Let $h_\tau(r)=\frac12(r-\tau)_+^2$. The dual functions constructed
below satisfy
\begin{equation}\label{eq:hinge-dual}
 \phi_\tau(x)+\psi_\tau(y)+\theta_\tau(x)(y-x)
 \leq h_\tau(y-x),
\end{equation}
for $x\in[-1,1]$ and $y\in I_k(x)$, with equality at $D(x)$ and
$U(x)$.
\end{lemma}

\begin{proof}
Since $h_\tau\in C^1$, define
\begin{equation*}
 S(x):=Q_x((\tau,\infty))\in[0,1]\quad(x>-1),
 \qquad S(-1):=0,
 \qquad \theta_\tau(x):=-\int_{-1}^xS(s)\dd s.
\end{equation*}
Define $\psi_\tau$ by \eqref{eq:psi-lower}--\eqref{eq:psi-upper},
using $h_\tau'(r)=(r-\tau)_+$, and define $\phi_\tau$ by
\eqref{eq:phi-def}. For $-1<x\leq b$,
\[
 S(x)=\frac{(w(x)-\tau)_+-(d(x)-\tau)_+}{w(x)-d(x)},
\]
and for $b\leq x\leq1$,
\[
 (w+k)S=(w-\tau)_+-(\zeta-\tau)_+
          +\int_{\alpha(x)}^xS(s)\dd s.
\]
These identities give both equalities in \cref{lem:contacts} by the
same absolutely-continuous calculation. The value assigned to $S(-1)$
does not affect the dual functions. Furthermore,
\cref{lem:stochastic-order}, applied to $\1_{(\tau,\infty)}$, gives
\[
 (w+k)S\geq\int_{-k}^w\1_{\{r>\tau\}}\dd r
 =h_\tau'(w)-h_\tau'(-k),
\]
so \eqref{eq:boundary-sign} holds as well.

Write $A(x)=\int_{-1}^xS(s)\dd s=-\theta_\tau(x)$ and
$g_x(y)=\partial_yG_x(y)$. Each target $y$ has a representation
$y=T(t)$, using $T=D|_{[-1,b]}$ for $y\leq-1$ and $T=U$ for
$y\geq-1$. The two representations agree at $y=-1$.
The definition of $\psi_\tau'$ gives, for either order of $t$ and $x$,
\begin{equation}\label{eq:g-formula}
\begin{split}
 g_x(T(t))
 &=\int_t^x
 \left[S(s)-\1_{\{T(t)-s>\tau\}}\right]\dd s.
\end{split}
\end{equation}
Indeed,
\[
 h_\tau'(T(t)-x)-h_\tau'(T(t)-t)
 =-\int_t^x\1_{\{T(t)-s>\tau\}}\dd s.
\]

Fix $x$ and first take $D(x)\leq y\leq U(x)$. The representing
source satisfies $t=t(y)\leq x$. Set
\[
 e(y):=y-t(y),\qquad r(y):=y-x,\qquad C(y):=A(t(y))+e(y).
\]
Both $e$ and $C$ are nondecreasing in $y$. To see this on $[-2,-1]$,
let $-1\leq s<t\leq b$. Since $D$ is strictly decreasing and
$0\leq S\leq1$,
\[
 d(t)-d(s)=D(t)-D(s)-(t-s)<0,
\]
and
\[
 [A(t)+d(t)]-[A(s)+d(s)]
 =D(t)-D(s)-\int_s^t(1-S(u))\dd u<0.
\]
The inverse $t(y)$ decreases there, proving the assertion. On $[-1,2]$,
$t(y)$ increases, and both $w(t)$ and $A(t)$ are nondecreasing.
At $y=-1$, both definitions give $e=C=0$.

Since $r(y)\leq e(y)$, \eqref{eq:g-formula} yields
\[
\begin{array}{lll}
 e(y)\leq\tau:
 &g_x(y)=A(x)-A(t)\geq0,\\[2mm]
 r(y)<\tau<e(y):
 &g_x(y)=A(x)+\tau-C(y),\\[2mm]
 \tau\leq r(y):
 &g_x(y)=A(x)-A(t)-(x-t)\leq0.
\end{array}
\]
These regimes occur in the displayed order as $y$ increases; at equality
the adjacent formulas agree. In the middle regime, $g_x$ is
nonincreasing. Therefore $g_x$ can change sign only from nonnegative
to nonpositive. Since $G_x(D(x))=G_x(U(x))=0$, this proves
$G_x\geq0$ on $[D(x),U(x)]$. For $x\geq b$, the left endpoint is
$B(x)$ and its right derivative is also nonnegative by
\eqref{eq:boundary-sign}.

Now suppose $x<b$ and $-2\leq y=D(t)<D(x)$. Then $t>x$, and
\[
 g_x(y)=\int_x^t
 \left[\1_{\{D(t)-s>\tau\}}-S(s)\right]\dd s\leq0.
\]
Indeed, $D(t)\leq D(s)$. If the indicator is one, then
$d(s)>\tau$ and \eqref{eq:Q-free} gives $S(s)=1$; if it is zero,
the integrand is $-S(s)\leq0$. Integrating from $y$ to $D(x)$
therefore gives $G_x(y)\geq G_x(D(x))=0$.

Finally, if $U(x)<y=U(t)\leq2$, then $t>x$ and
\[
 g_x(y)=\int_x^t
 \left[\1_{\{U(t)-s>\tau\}}-S(s)\right]\dd s\geq0.
\]
Here $U(t)\geq U(s)$. If the indicator is zero, then
$w(s)\leq U(t)-s\leq\tau$, and \eqref{eq:Q-support} gives
$S(s)=0$; otherwise the integrand is $1-S(s)\geq0$.
Thus $G_x(y)\geq G_x(U(x))=0$. For $x\geq b$, no admissible target
lies below $D(x)=B(x)$. This covers $I_k(x)$ and proves
\eqref{eq:hinge-dual}.
\end{proof}

\begin{theorem}[Optimality of the explicit coupling]
\label{thm:optimality}
Let $2\leq k<3$ and let $\pi_k$ be defined by \eqref{eq:pi-k}. For
every $h\in C^1([-3,3])$ with convex derivative, $\pi_k$ minimizes
\eqref{eq:primal}.
\end{theorem}

\begin{proof}
For $h\in C^3([-3,3])$ with $h'''\geq0$, Taylor's formula gives
\begin{equation*}
 h(r)=\alpha+\beta r+\gamma r^2
      +\frac12\int_{-3}^{3}(r-\tau)_+^2h'''(\tau)\dd\tau
\end{equation*}
for suitable constants $\alpha,\beta,\gamma$.
The first three terms have the same expectation under every admissible
coupling, because
\[
 \mathbb E[Y-X]=0,\qquad
 \mathbb E[(Y-X)^2]=\mathbb E[Y^2]-\mathbb E[X^2].
\]
They also admit exact dual representations; for the quadratic term,
\[
 (y-x)^2=y^2-x^2-2x(y-x).
\]
Integrating \eqref{eq:hinge-dual} with the nonnegative weight
$h'''(\tau)\dd\tau$ and adding these representations gives
\begin{equation}\label{eq:dual-inequality}
 \phi(x)+\psi(y)+\theta(x)(y-x)\leq h(y-x)
\end{equation}
for $x\in[-1,1]$ and $y\in I_k(x)$, with equality on the two graphs
supporting $\pi_k$. The hinge dual functions are jointly measurable
and bounded uniformly for $\tau\in[-3,3]$ and for the source and target
variables in their compact intervals. This justifies the integration.
For every $\pi\in\M_k(\mu,\nu)$, the martingale property gives
\[
 \int h(y-x)\dd\pi
 \geq\mu(\phi)+\nu(\psi)
 =\int h(y-x)\dd\pi_k.
\]
For $h\in C^1([-3,3])$ with convex derivative, apply this result to
$h_n$ from \eqref{eq:convex-approximation}. Since
\[
 \sup_{\pi\in\M_k(\mu,\nu)}
 \left|\int(h_n-h)(y-x)\dd\pi\right|
 \leq\|h_n-h\|_\infty\longrightarrow0,
\]
the optimality inequality passes to the limit.
\end{proof}

\subsection{Optimality for
\texorpdfstring{$1<k<2$}{1 < k < 2}}\label{sec:small-optimality}

We extend the dual proof to $1<k<2$, using estimates for
$\partial_xG_x$ and $\partial_yG_x$ to exclude negative minima of $G_x$.

Fix $1<k<2$ and let $D,U$ be the maps of \cref{prop:small-k}.
We retain the notation $d=D-x$, $w=U-x$, $B=x-k$, and $G_x$
from the preceding subsection. Set $v=x-D$ on $[-1,b]$.
The proof reduces to the costs $h_\tau(r)=\tfrac12(r-\tau)_+^2$.
For these costs, we will show that every negative global minimum of
$y\mapsto G_x(y)$ must belong to
\[
 J_\tau:=\{U(t):x_0<t<b,\ v(t)<\tau<w(t)\}.
\]
This set is either empty or an open interval. We prove nonnegativity
there by combining a comparison in $x$ with a sign estimate for
$\partial_yG_x$; see \cref{fig:comparison-geometry}.

For $x\in[b,1]$, the point $B(x)$ belongs to the image of
$D|_{[-1,b]}$ if $x\leq k-1$, and to the image of $U|_{[-1,x_0]}$
if $x\geq k-1$. Define
\begin{equation}\label{eq:small-predecessor}
 \alpha(x)=
 \begin{cases}
 D_{\mathrm f}^{-1}(x-k),&b\leq x\leq k-1,\\
 U_{\mathrm e}^{-1}(x-k),&k-1\leq x\leq1,
 \end{cases}
 \qquad \zeta(x)=x-k-\alpha(x),
\end{equation}
where $D_{\mathrm f}:=D|_{[-1,b]}$ and
$U_{\mathrm e}:=U|_{[-1,x_0]}$. These restrictions are strictly
monotone. Both inverse formulas give $\alpha(k-1)=-1$, and
\[
 \alpha(b)=b,\qquad \alpha(1)=x_0,\qquad
 x-\alpha(x)=k+\zeta(x).
\]
Thus $\alpha(x)$ is the unique source in the indicated interval whose
lower or upper image equals $B(x)$; see \cref{fig:predecessor-geometry}.

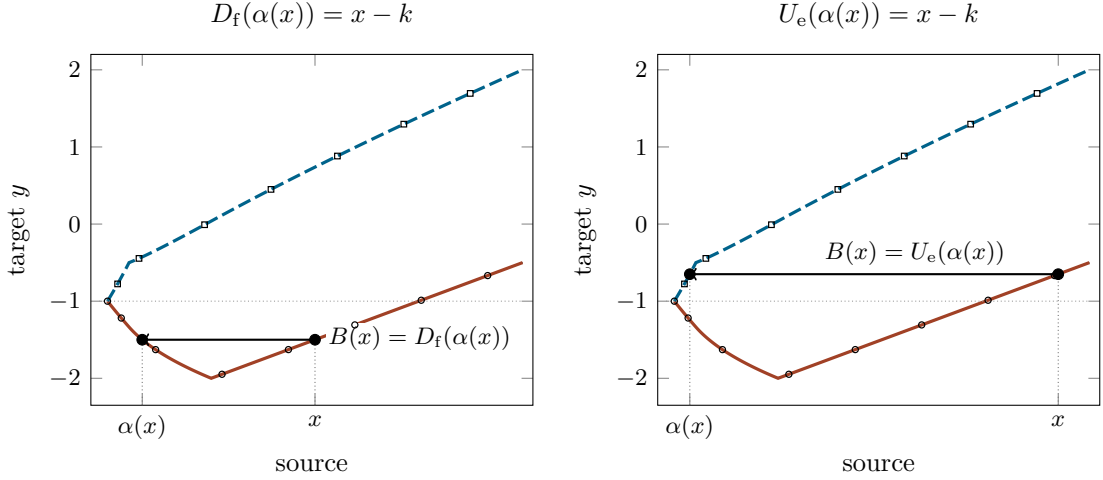
\begin{figure}[tbp]
\centering
\begin{tikzpicture}
\begin{groupplot}[
 group style={group size=2 by 1,horizontal sep=0.65in},
 width=0.45\textwidth,height=2.45in,scale only axis=false,
 xmin=-1.08,xmax=1.05,ymin=-2.35,ymax=2.2,
 xlabel={source},ylabel={target $y$},
 xtick=\empty,ytick={-2,-1,0,1,2},
 tick label style={font=\footnotesize},title style={font=\small},
 label style={font=\small},clip=false]
\nextgroupplot[title={$D_{\mathrm f}(\alpha(x))=x-k$},
 xtick={\TraceLowerAlpha,\TraceLowerX},xticklabels={$\alpha(x)$,$x$}]
\addplot[support obstacle,domain=-0.5:1] {x-1.5};
\addplot[constrained lower,mark repeat=80] table[x=s,y=D] {predecessor_support.dat};
\addplot[constrained upper,mark repeat=80] table[x=s,y=U] {predecessor_support.dat};
\draw[black!40,densely dotted] (axis cs:-1.04,-1)--(axis cs:1,-1);
\draw[black!55,densely dotted]
 (axis cs:\TraceLowerAlpha,-2.35)--(axis cs:\TraceLowerAlpha,\TraceLowerY)
 (axis cs:\TraceLowerX,-2.35)--(axis cs:\TraceLowerX,\TraceLowerY);
\draw[->,black,thick] (axis cs:\TraceLowerX,\TraceLowerY)--(axis cs:\TraceLowerAlpha,\TraceLowerY);
\addplot[only marks,mark=*,mark size=2pt,black]
 coordinates {(\TraceLowerAlpha,\TraceLowerY) (\TraceLowerX,\TraceLowerY)};
\node[font=\footnotesize,anchor=west,fill=white,inner sep=1pt]
 at (axis cs:0.05,-1.48) {$B(x)=D_{\mathrm f}(\alpha(x))$};
\nextgroupplot[title={$U_{\mathrm e}(\alpha(x))=x-k$},
 xtick={\TraceUpperAlpha,\TraceUpperX},xticklabels={$\alpha(x)$,$x$}]
\addplot[support obstacle,domain=-0.5:1] {x-1.5};
\addplot[constrained lower,mark repeat=80] table[x=s,y=D] {predecessor_support.dat};
\addplot[constrained upper,mark repeat=80] table[x=s,y=U] {predecessor_support.dat};
\draw[black!40,densely dotted] (axis cs:-1.04,-1)--(axis cs:1,-1);
\draw[black!55,densely dotted]
 (axis cs:\TraceUpperAlpha,-2.35)--(axis cs:\TraceUpperAlpha,\TraceUpperY)
 (axis cs:\TraceUpperX,-2.35)--(axis cs:\TraceUpperX,\TraceUpperY);
\draw[->,black,thick] (axis cs:\TraceUpperX,\TraceUpperY)--(axis cs:\TraceUpperAlpha,\TraceUpperY);
\addplot[only marks,mark=*,mark size=2pt,black]
 coordinates {(\TraceUpperAlpha,\TraceUpperY) (\TraceUpperX,\TraceUpperY)};
\node[font=\footnotesize,anchor=south,fill=white,inner sep=1pt]
 at (axis cs:0.16,-0.57) {$B(x)=U_{\mathrm e}(\alpha(x))$};
\end{groupplot}
\end{tikzpicture}
\caption{The inverse relation defining $\alpha(x)$ for $k=3/2$.
Left: $D_{\mathrm f}(\alpha(x))=x-k$ for $b\leq x<k-1$.
Right: $U_{\mathrm e}(\alpha(x))=x-k$ for $k-1<x\leq1$.
At $x=k-1$, both definitions give $\alpha(x)=-1$.
The horizontal arrows join points with equal target coordinate.}
\label{fig:predecessor-geometry}
\end{figure}

Define $Q_x$ by \eqref{eq:Q-free}--\eqref{eq:Q-active}, using
\eqref{eq:small-predecessor}, and define the dual functions by
\eqref{eq:S-theta}--\eqref{eq:phi-def}. The proof of
\cref{lem:contacts} still applies: evaluate $\psi'(B(x))$ by
\eqref{eq:psi-lower} when $B(x)\leq-1$, and by
\eqref{eq:psi-upper} when $B(x)>-1$. Hence $G_x$ vanishes at both
support points. We next establish the estimates needed to prove
$G_x\geq0$.

\begin{lemma}[The kernel for $1<k<2$]
\label{lem:small-kernel}
Equations \eqref{eq:Q-free}--\eqref{eq:Q-active} define a unique
probability kernel $(Q_x)_{x\in[-1,1]}$, with
\[
 \supp Q_x\subset[-k,w(x)].
\]
For $-1\leq x\leq k-1$, it also satisfies
\[
 Q_x\leq_{\mathrm{st}}\Unif[0,w(x)],
\]
where the law at $w=0$ is interpreted as $\delta_0$.
\end{lemma}

\begin{proof}
We first bound $w(s)$ for the sources $s$ in the integral defining $Q_x$.
For $-1<x\leq x_0$, strict convexity of $y(\delta)$ gives
\[
 U'(x)=\frac{2v}{v+w}y'(R(x))
       >\frac{2v}{R(x)-L(x)}
       =\frac{v}{x+1}>1,
\]
because $D(x)<-1$. Thus $w$ is strictly increasing on $[-1,x_0]$.
On $[x_0,b]$, \eqref{eq:middle-gaps} gives
$w'=C/(2W+C)$. Since $C$ increases with the source coordinate,
$w$ first decreases and then increases, with either part possibly empty.
Moreover, for $\delta_b<\delta\leq\delta_0$,
\[
 \delta^2-W(\delta)^2>0.
\]
Indeed, the difference $\delta^2-\mathcal W(\delta)$ vanishes at
$\delta_b$ and has derivative
$2\delta+C=\delta+2k-2-z(\delta)$. This derivative is decreasing on
$[\delta_b,1]$ and equals $2k-2>0$ at $\delta=1$.

Fix $x\in[b,1]$ and set $\delta=w(x)$.
If $\alpha(x)\in[x_0,b]$, its parameter in
\eqref{eq:middle-delta} equals $\delta$. Therefore
$\max_{\alpha(x)\leq s\leq b}w(s)$ is
$\max\{W(\delta),\delta_b\}\leq\delta$.
If $\alpha(x)\in[-1,x_0]$, then $\delta\geq\delta_0$, whereas
$\max_{-1\leq s\leq b}w(s)$ is
$\max\{w(x_0),\delta_b\}<\delta_0$. Here $w(x_0)<\delta_0$ because
$\delta_0-w(x_0)=(\delta_0-3+2k)/2$ and $3-2k<\delta_b<\delta_0$. The
first inequality is trivial if $k\geq3/2$; otherwise it follows from
$z(3-2k)=2k-1-2k\log3<k-2=z(\delta_b)$, as $2k\log3>2k>k+1$.
Since $w$ is increasing on $[b,1]$, for every $x\in[b,1]$ we obtain
\begin{equation}\label{eq:small-window}
 \zeta(x)\leq w(x),\qquad
 w(s)\leq w(x)\quad\text{for }\alpha(x)\leq s\leq x.
\end{equation}
Positivity, total mass one, and uniqueness now follow from the same
Volterra iteration as in \cref{sec:duality}.
Using \eqref{eq:small-window} at each iterate proves the support assertion.

For the stochastic-order assertion, it suffices to compare upper tails
at $\tau\geq0$. For $x\in[-1,b]$, the assertion follows from
$d\leq0$. For $b\leq x\leq k-1$, one has $\zeta\leq0$ and
$x-\alpha(x)\leq k$.
If the preceding kernels satisfy the asserted upper bound, their tails,
by \eqref{eq:small-window}, are at most $(w(x)-\tau)_+/w(x)$.
The numerator in \eqref{eq:Q-active} is therefore at most
\[
 (w(x)-\tau)_+
 +(x-\alpha(x))\frac{(w(x)-\tau)_+}{w(x)}
 \leq (k+w(x))\frac{(w(x)-\tau)_+}{w(x)}.
\]
The Volterra iteration preserves this bound and proves the claim.
\end{proof}

The remaining estimates use the values of $v$ and $w$ at $x_0$. Write
\[
 p=k-1,\quad y_0=1-k,\quad
 v_0=v(x_0),\quad w_0=w(x_0),\quad
 \ell=x_0+1=\frac{1-\delta_0}{2}.
\]
Since $D$ decreases on $[-1,b]$, $v$ is strictly increasing there.
The following elementary comparison will be applied with $J=J_\tau$.

\begin{lemma}[Comparison in the source coordinate]
\label{lem:source-comparison}
Let $(x,y)\mapsto G_x(y)$ be continuously differentiable on
$b\leq x\leq1$, $B(x)\leq y\leq U(x)$, with
$G_x(B(x))=G_x(U(x))=0$. Suppose that for every fixed $x$,
each negative global minimum of $y\mapsto G_x(y)$ lies in an open
interval $J$ independent of $x$, with
$\overline J\subset[B(1),U(b)]$. Assume also $G_b\geq0$.
If there is $x_*\in[b,1]$ such that
\begin{enumerate}[label=\textup{(\roman*)},leftmargin=2.2em]
 \item $\partial_yG_x(y)\geq0$ on $J$ for $x\geq x_*$;
 \item for each $y\in J$, $x\mapsto G_x(y)$ attains its minimum on
       $[b,x_*]$ at an endpoint,
\end{enumerate}
then $G_x(y)\geq0$ throughout its domain.
\end{lemma}

\begin{proof}
Since $B$ and $U$ are increasing,
$\overline J\subset[B(x),U(x)]$ for every $x\in[b,1]$.
For $x\geq x_*$, monotonicity on $J$ and continuity would extend any
negative global minimum in $J$ to its left endpoint, contrary to the
hypothesis. Thus $G_x\geq0$ for these $x$.
For $b\leq x\leq x_*$ and $y\in J$, condition~(ii) gives
$G_x(y)\geq\min\{G_b(y),G_{x_*}(y)\}\geq0$. The assumed location
of every negative minimum then implies $G_x\geq0$ on its full domain.
\end{proof}

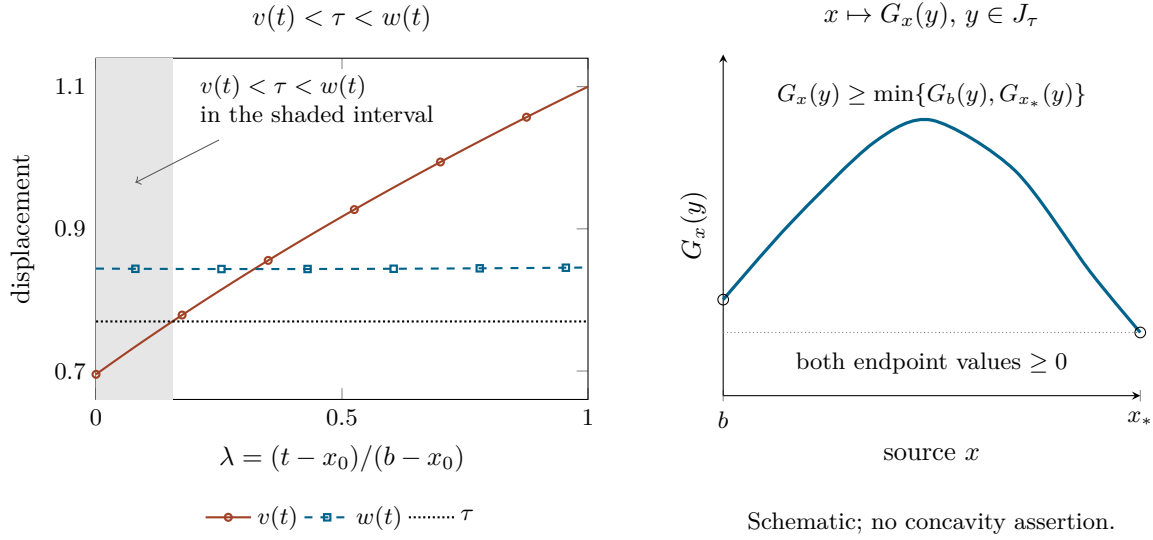
\begin{figure}[tbp]
\centering
\begin{tikzpicture}
\begin{axis}[
 name=displacementpanel,width=0.49\textwidth,height=2.4in,
 xmin=0,xmax=1,ymin=0.66,ymax=1.14,
 xlabel={$\lambda=(t-x_0)/(b-x_0)$},ylabel={displacement},
 title={$v(t)<\tau<w(t)$},
 title style={font=\small},label style={font=\small},
 tick label style={font=\footnotesize},
 xtick={0,0.5,1},ytick={0.7,0.9,1.1},
 legend style={font=\footnotesize,draw=none,at={(0.5,-0.28)},anchor=north,legend columns=3},
 clip=false]
\fill[black!10] (axis cs:0,0.66) rectangle (axis cs:\ComparisonEnd,1.14);
\addplot[BrickRed!85!black,thick,mark=o,mark repeat=35,
 mark options={solid,fill=white},mark size=1.3pt]
 table[x=lambda,y=v] {middle_displacements.dat};
\addlegendentry{$v(t)$}
\addplot[MidnightBlue!90!black,thick,dashed,mark=square,mark repeat=35,
 mark phase=17,mark options={solid,fill=white},mark size=1.2pt]
 table[x=lambda,y=w] {middle_displacements.dat};
\addlegendentry{$w(t)$}
\addplot[black,densely dotted,thick,domain=0:1] {\ComparisonTau};
\addlegendentry{$\tau$}
\node[font=\footnotesize,anchor=north west,align=left,fill=white,inner sep=2pt]
 at (axis cs:0.2,1.125) {$v(t)<\tau<w(t)$\\in the shaded interval};
\draw[black!65,->] (axis cs:0.25,1.025)--(axis cs:0.08,0.965);
\end{axis}
\begin{axis}[
 at={(displacementpanel.outer east)},anchor=outer west,xshift=0.35in,
 width=0.43\textwidth,height=2.4in,
 xmin=0,xmax=1,ymin=0,ymax=1.35,
 axis lines=left,xtick={0,1},xticklabels={$b$,$x_*$},ytick=\empty,
 xlabel={source $x$},ylabel={$G_x(y)$},
 title={$x\mapsto G_x(y)$, $y\in J_\tau$},
 title style={font=\small},label style={font=\small},
 tick label style={font=\footnotesize},clip=false]
\addplot[black!50,densely dotted,domain=0:1] {0.25};
\addplot[MidnightBlue!90!black,very thick,smooth]
 coordinates {(0,0.38) (0.16,0.68) (0.36,1.0) (0.50,1.09)
              (0.70,0.89) (0.88,0.49) (1,0.25)};
\addplot[only marks,mark=o,mark size=2pt,black,mark options={fill=white}]
 coordinates {(0,0.38) (1,0.25)};
\node[font=\footnotesize,anchor=south,align=center] at (axis cs:0.5,1.1)
 {$G_x(y)\geq\min\{G_b(y),G_{x_*}(y)\}$};
\node[font=\footnotesize,anchor=north,align=center] at (axis cs:0.5,0.21)
 {both endpoint values $\geq0$};
\node[font=\footnotesize,anchor=north,align=center] at (axis cs:0.5,-0.42)
 {Schematic; no concavity assertion.};
\end{axis}
\end{tikzpicture}
\caption{Left: $v(t)=t-D(t)$ and $w(t)=U(t)-t$ on $[x_0,b]$
for $k=1.1$ and $\tau\approx0.76972$.
The shaded set satisfies $v(t)<\tau<w(t)$, and its image under $U$
is $J_\tau$. Right (schematic): for fixed $y\in J_\tau$,
$G_x(y)\geq\min\{G_b(y),G_{x_*}(y)\}$ on $[b,x_*]$.
The proof establishes this bound by derivative signs; the plotted
curve is not a computed dual gap.}
\label{fig:comparison-geometry}
\end{figure}

\begin{lemma}[Estimates at $x_0$]
\label{lem:small-dual-estimates}
Suppose $v_0<w_0$. Then
\begin{gather*}
 1<k<\frac{67}{50},\qquad
 \frac35<\frac{1-\delta_0}{k-1}<\frac75,\qquad U(b)<0,\\
 w(x)-\tau>\frac{k-w(x)}4
 \quad\text{if }x\in[b,1],\ x\geq y_0,\ v_0\leq\tau\leq w_0.
\end{gather*}
Furthermore,
\begin{equation}\label{eq:small-activation-bound}
 (\delta_0-\tau)
 \left[\exp\left(\frac{\tau+v_0-\ell}{k+1}\right)-1\right]
 \geq w_0-\tau,\qquad v_0\leq\tau\leq w_0.
\end{equation}
\end{lemma}

\begin{proof}
\emph{1. Bounds on $k$ and $\delta_0$.}
The zero of $C$ is
$\delta_m=k-(k-1)e^{(2-k)/k}\in(\delta_b,1)$.
Put $r=(k-1)/k$ and $E=e^{1-2r}$. Substitution into
\eqref{eq:T-delta} gives
\[
 \frac{4T(\delta_m)}{k^2}=r f(r),\qquad
 f(r)=r(3+E)^2-4(3-E).
\]
Differentiation gives
\[
 f''(r)=8E\{r(3+2E)-(1+E)\}.
\]
The expression in braces is strictly increasing, since its derivative
is $3+4E(1-r)>0$. Thus $f'$ first decreases and then increases,
with either part possibly empty. Since $f'(1/2)=0$, $f'$ cannot
cross from negative to positive before the right endpoint. Hence
$f$ has no interior minimum on
$[17/67,1/2]$. At the right endpoint $f=0$.
At the left endpoint,
$e^{33/67}>81823/50000$, and substitution of this rational lower bound
gives
\[
 f(17/67)>\frac{40556593}{167500000000}>0.
\]
The exponential bound follows already from its Taylor polynomial of
degree eight. It follows that $T(\delta_m)\geq0$ when $k\geq67/50$.
By \cref{lem:small-k-transition}, this implies
$\delta_0\leq\delta_m$ and $C(\delta_0)\geq0$, or $v_0\geq w_0$.
The hypothesis therefore forces $0<p<17/50$.

For $t>0$, direct substitution gives
\[
 T(1-pt)=pt\{1-p+pt/4-2(1+p)H(t)\},\qquad
 H(t)=\frac{1+t}{t}\log(1+t)-1.
\]
At $t=3/5$ the expression in braces decreases with $p$.
The bound $\log(8/5)<47001/100000$ shows that its value at
$p=17/50$ exceeds $19997/625000>0$.
Hence $\delta_0<1-3p/5$.
At $t=7/5$, the bound $\log(12/5)>7/8$ gives $H(t)>1/2$,
so that the expression in braces is negative. Hence
$\delta_0>1-7p/5$. These logarithm bounds follow from the positive series
for $\log z$ in powers of $(z-1)/(z+1)$, with a geometric bound on
the remainder.

\medskip\noindent\emph{2. Bounds on $w$ and $U(b)$.}
The bound on $\delta_0$ gives $w_0=(3-2k+\delta_0)/2<1-13p/10$.
Also,
\[
 z(1-5p/6)=1+5p/6-2(1+p)\log(11/6)<-p
 \quad(0<p\leq17/50),
\]
using $\log(11/6)>487/804$.
Monotonicity of $z$ yields, for $x\in[b,1]$ with $x\geq y_0=-p$,
\[
 w(x)>1-\frac{5p}{6},\qquad
 w(x)-\tau\geq w(x)-w_0>\frac{7p}{15}>
 \frac{11p}{24}>\frac{k-w(x)}4.
\]
Finally $z(2-k)-b=2(1-k\log2)>0$ in this range, so that
$\delta_b<2-k$ and $U(b)=b+\delta_b<0$.

\medskip\noindent\emph{3. Reduction of \eqref{eq:small-activation-bound} to a rational inequality.}
Set $t=(1-\delta_0)/p\in[3/5,7/5]$. The equation $T(\delta_0)=0$
expresses all quantities in terms of this single variable:
\begin{equation}\label{eq:small-scalar-parameters}
\begin{split}
 H&=\frac{1+t}{t}\log(1+t)-1,\qquad
 p=\frac{1-2H}{1+2H-t/4},\\
 \delta_0&=1-pt,\qquad \ell=pt/2,\\
 v_0&=p-1-pt/2+2(1+p)\log(1+t),\qquad
 w_0=1-p-pt/2.
\end{split}
\end{equation}
In particular, $\delta_0-w_0=p(1-t/2)>0$.
Also $v_0>\ell$, since $D(x_0)<-1$ and $\ell=x_0+1$.
Define, for $\tau\in[v_0,w_0]$,
\[
 \mathcal E(\tau)
 =(\delta_0-\tau)
 \left[\exp\left(\frac{\tau+v_0-\ell}{k+1}\right)-1\right]
 -(w_0-\tau).
\]
Direct differentiation gives
\[
 \mathcal E''(\tau)
 =\frac{\exp((\tau+v_0-\ell)/(k+1))}{(k+1)^2}
   \bigl(\delta_0-\tau-2(k+1)\bigr)<0,
\]
because $\delta_0<1$, $\tau\geq v_0>\ell>0$, and $k+1>2$.
Moreover $\mathcal E(w_0)>0$, since $\delta_0>w_0$ and
$w_0+v_0-\ell>0$. By concavity, it suffices to prove
$\mathcal E(v_0)\geq0$.

Using $2(1+p)H=1-p+pt/4$ in
\eqref{eq:small-scalar-parameters} eliminates the logarithm from $v_0$:
\[
 v_0=\frac{p-1+(2+3p/2)t-pt^2/4}{1+t}.
\]
Put $a=(2v_0-\ell)/(2+p)$. The tangent inequality for the exponential
at $\log2$ gives, for every real $a$,
\[
 e^a-1\geq2a+1-2\log2>2a-\frac25.
\]
Here $\log2<7/10$ follows, for example, from
\[
 e^{7/10}>1+\frac7{10}+\frac{49}{200}+\frac{343}{6000}
 =\frac{12013}{6000}>2.
\]
Since $\delta_0-v_0>0$, it is therefore enough to prove the following
rational inequality:
\begin{equation}\label{eq:small-rational-bound}
 (\delta_0-v_0)\left(2a-\frac25\right)-(w_0-v_0)
 =\frac{N(p,t)}{20(p+2)(1+t)^2}>0,
\end{equation}
where direct expansion yields
\[
 \begin{split}
 N(p,t)={}&(30t^4+36t^3-179t^2-172t-32)p^2\\
 &+(-58t^3-310t^2+388t+280)p\\
 &-104t^2+344t-272.
 \end{split}
\]

\medskip\noindent\emph{4. Positivity of $N(p,t)$.}
Set $q=t/(2+t)\in(0,1)$.
The positive series for the logarithm gives
\[
 2q\leq\log(1+t)
 =2\sum_{j=0}^{\infty}\frac{q^{2j+1}}{2j+1}
 \leq2q+\frac{2q^3}{3(1-q^2)}.
\]
Consequently
\[
 \frac{t}{t+2}\leq H\leq\frac{t(t+6)}{6(t+2)}.
\]
The function $H\mapsto(1-2H)/(1+2H-t/4)$ is decreasing on this
interval, with positive denominator. Substitution gives
\[
 p_-(t):=\frac{4(6-3t-t^2)}{t^2+30t+24}
 \leq p\leq
 p_+(t):=\frac{4(2-t)}{8+10t-t^2}.
\]
All denominators here are positive for $3/5\leq t\leq7/5$.
As a polynomial in $p$, $N(p,t)$ is concave: its leading coefficient is
\[
 t^2(30t^2+36t-179)-172t-32<0,
\]
since $30t^2+36t-179\leq-349/5$ on this interval.
It suffices, therefore, to prove positivity at $p_-(t)$ and $p_+(t)$.

At the upper endpoint, direct substitution gives
\[
 N(p_+(t),t)
 =\frac{16(t-2)(t+1)^2(9t^3+164t^2-316t+48)}
 {(t^2-10t-8)^2}>0.
\]
Indeed, the cubic in the numerator is convex on $[3/5,7/5]$ and
has endpoint values $-10077/125$ and $-6033/125$. It is thus negative
throughout the interval, as is $t-2$.
At the lower endpoint,
\[
 N(p_-(t),t)=\frac{8P(t)}{(t^2+30t+24)^2},
\]
where
\[
 \begin{split}
 P(t)={}&60t^8+461t^7+993t^6+2448t^5-4198t^4\\
 &-10356t^3+5400t^2+8784t-1728.
 \end{split}
\]
To prove $P(t)>0$ on $[3/5,7/5]$, put $u=t-1$. Then
\[
 \begin{split}
 P(1+u)={}&60u^8+941u^7+5900u^6+21447u^5+43272u^4\\
 &+36687u^3-120u^2-6371u+1864.
 \end{split}
\]
If $-2/5\leq u\leq0$, put $r=-u\in[0,2/5]$ and group the terms as
\[
 \begin{split}
 P(1-r)={}&60r^8+r^6(5900-941r)+r^4(43272-21447r)\\
 &+r(6371-36687r^2)+1864-120r^2
 \geq\frac{9224}{5}>0.
 \end{split}
\]
Each parenthesized coefficient is positive on $[0,2/5]$.
If $0\leq u\leq1/4$, all terms of degree at least three are
nonnegative, so
\[
 P(1+u)\geq1864-6371u-120u^2
 \geq1864-\frac{6371}{4}-\frac{120}{16}
 =\frac{1055}{4}>0.
\]
Finally, if $1/4\leq u\leq2/5$, then $u^3\geq u/16$, and hence
\[
 P(1+u)\geq1864-\frac{65249}{16}u-120u^2
 \geq1864-\frac{65249}{40}-\frac{96}{5}
 =\frac{8543}{40}>0.
\]
Thus $N$ is positive at both rational endpoints, and concavity in $p$
proves \eqref{eq:small-rational-bound}. The exponential bound gives
$\mathcal E(v_0)>0$; concavity in $\tau$ now proves
\eqref{eq:small-activation-bound} throughout $[v_0,w_0]$.
\end{proof}

\begin{theorem}[Optimality for $1<k<2$]\label{thm:small-optimality}
Let $1<k<2$ and let $\pi_k$ be the coupling of \cref{prop:small-k}.
Then, for every $h\in C^1([-3,3])$ with convex derivative,
\[
 V_k(h)=\int h(y-x)\,\pi_k(\dd x,\dd y).
\]
For $h\in C^3([-3,3])$ with $h'''\geq0$, the dual functions defined above satisfy
\eqref{eq:dual-inequality}, with equality on the support of $\pi_k$.
\end{theorem}

\begin{proof}
It suffices to prove the dual inequality for
$h_\tau(r)=\tfrac12(r-\tau)_+^2$.
Write
\begin{gather*}
 f(r)=(r-\tau)_+,\qquad S(x)=Q_x((\tau,\infty)),\qquad
 F(x)=\int_{-1}^xS(s)\dd s,\\
 g_x(y)=\partial_yG_x(y),\qquad K(x,y)=\partial_xG_x(y).
\end{gather*}
The Volterra equation and the contact calculation give
\begin{equation}\label{eq:small-S-gap}
\begin{split}
 (k+w)S&=(w-\tau)_+-(\zeta-\tau)_+
               +F(x)-F(\alpha(x)),\\
 K(x,y)&=f(w)-f(y-x)-S(x)(U(x)-y).
\end{split}
\end{equation}
The first identity holds for $x\in[b,1]$. The second follows by
differentiating the contact equality on each smooth interval with
$x>-1$ and then extending by continuity at the transitions. At $x=-1$
we use its integrated form. All further differential identities are
understood on the smooth intervals and extended by integration.

\medskip\noindent\emph{1. Sources $x\in[-1,b]$ and targets outside $[D(x),U(x)]$.}
For $x\in[-1,b]$, the intervals $[D(x),U(x)]$ increase with $x$.
Fix $y\in[D(x),U(x)]$, and let $t\leq x$ be the unique source with
$D(t)=y$ if $y\leq-1$, or $U(t)=y$ if $y\geq-1$.
Then $G_t(y)=0$ and $y\in[D(s),U(s)]$ for $t\leq s\leq x$.
On $[-1,b]$, $S$ is the chord slope of the convex function $f$
on $[d,w]$. Thus, for $r=y-s\in[d(s),w(s)]$,
$f(r)\leq f(w)-S(w-r)$, and \eqref{eq:small-S-gap} gives
$K(s,y)\geq0$. Integration from $t$ to $x$ proves $G_x(y)\geq0$.
The sign estimates outside $[D(x),U(x)]$ are those of
\cref{lem:hinge}; for $y>U(x)$ they use the support bound in
\cref{lem:small-kernel}. This proves the inequality for every
$x\in[-1,b]$ and for $y\geq U(x)$ when $x\in[b,1]$.

\medskip\noindent\emph{2. Location of a negative minimum for $x\in[b,1]$.}
Fix $x\in[b,1]$ and $B(x)<y<U(x)$. Define
\[
 t=
 \begin{cases}
 D_{\mathrm f}^{-1}(y),&y\leq-1,\\
 U^{-1}(y),&y\geq-1.
 \end{cases}
\]
Both expressions give $t=-1$ at $y=-1$, and $t<x$.
Formula \eqref{eq:g-formula} becomes
\begin{equation}\label{eq:small-target-derivative}
 g_x(y)=F(x)-F(t)-(y-t-\tau)_+ +(y-x-\tau)_+.
\end{equation}
Put $e=y-t$ and $r=y-x$. If $\tau\leq0$, the regions
$e\leq\tau$, $r<\tau<e$, and $r\geq\tau$ occur in this order as $y$
increases: on the lower branch $e=d(t(y))$ increases, while on the
upper branch $e=w(t)\geq0\geq\tau$.
The sign argument of \cref{lem:hinge} then uses the following
derivative estimate in place of monotonicity of $w$:
for $t\in[x_0,b]$ with $d(t)<\tau<w(t)$,
\[
 (F+w)'=S+w'=\frac{v-\tau}{v+w}>0,
\]
and when $\tau\leq d$, it equals $U'>0$.
On $(-1,x_0)$ and $(b,1)$, $w'>0$.
Consequently $g_x$ can change sign only from positive to negative.
Since $G_x$ vanishes at $B(x)$ and $U(x)$, it is nonnegative between
these points. If $\tau\geq1$, then $S=0$ and the
normalized dual functions $\theta,\psi,\phi$ all vanish. Nonnegativity
of $h_\tau$ gives the dual inequality directly.
It remains to consider $0<\tau<1$.
If $e\leq\tau$, then
$g_x(y)=F(x)-F(t)\geq0$; if $r\geq\tau$, then
$g_x(y)=\int_t^x(S-1)\dd s<0$.
The inequality is strict: $Q_s((-\infty,0))>0$ for every $s>-1$.
For $s\leq b$ this follows from $d(s)<0$. For $s>b$, the integral in
\eqref{eq:Q-active} includes $(\alpha(s),b)$, where the kernels
have positive mass on $(-\infty,0)$. Hence $S(s)<1$.
If $r<\tau<e$, then $y=U(t)$, since $y=D(t)$ would give
$e=d(t)<0<\tau$. In this case,
\[
 g_x(U(t))=F(x)-F(t)-w(t)+\tau.
\]
On $(-1,x_0)$ and $(b,1)$, $(F+w)'>0$.
For $t\in(x_0,b)$, the identity $U'=2v/(v+w)$ gives
\begin{equation}\label{eq:small-middle-curvature}
 (F+w)'=\frac{v-\tau}{v+w},\qquad
 \partial_y g_x(U(t))=\frac{\tau-v(t)}{2v(t)}.
\end{equation}
Equation \eqref{eq:small-middle-curvature} implies that a negative
minimum away from the transition targets can occur only at $y=U(t)$
with $x_0<t<b$ and
\[
 v(t)<\tau<w(t).
\]
These are exactly the targets in $J_\tau$. Since $v$ increases
and $w$ decreases until $v=w$, the set is empty unless
$v_0<\tau<w_0$; otherwise it is an open interval with left endpoint
$y_0=U(x_0)=1-k$. Once either inequality fails, it cannot hold again.
Moreover, $F(t)+w(t)-\tau$ strictly decreases as $U(t)$ runs
through $J_\tau$.

We must also exclude the boundaries of the preceding cases. If
$e\leq\tau$ and $g_x(y)=0$, then $S=0$ almost everywhere between
$t$ and $x$. Equations \eqref{eq:Q-free}--\eqref{eq:Q-active} then force
$w(s)\leq\tau$ almost everywhere on this interval: for $s>b$,
the nonnegative Lebesgue term has endpoints $\zeta(s)<w(s)$,
and for $s\leq b$ the kernel is uniform on $[d(s),w(s)]$.
Also $y-s\leq y-t=e\leq\tau$. Thus $K(s,y)=0$ almost everywhere,
and integration from the contact at $t$ gives $G_x(y)=G_t(y)=0$.
This also excludes a negative minimum when $e=\tau$ or when $g_x$
vanishes on an interval.

For the transition targets with $r<\tau<e$, use the one-sided
signs of $g_x$. At $y_0$, the derivative $g_x$ is strictly decreasing
immediately to the left; if $g_x(y_0)=0$, the gap has smaller values
on that side. At $U(b)$, the same argument uses strict decrease of
$g_x$ immediately to the right. At an endpoint with $v(t)=\tau$,
$g_x$ first increases and then decreases. If it vanishes at that
endpoint, it is negative on the right, so the gap again has smaller
nearby values. None of these points can be a negative minimum.
Together with the strict sign for $r\geq\tau$, these observations
show that every negative global minimum of $G_x$ lies in $J_\tau$.
In particular, $J_\tau=\varnothing$ implies $G_x\geq0$.

\medskip\noindent\emph{3. The minimum in $x$ on $[b,x_*]$.}
Suppose now that $v_0<\tau<w_0$.
Let $\alpha_*\in(-1,x_0)$ be the unique solution of
$w(\alpha_*)=\tau$, and put
\[
 x_*=\alpha_*+k+\tau.
\]
Then $x_*\in(k-1,1)$ and $\zeta(x_*)=\tau$.
For $b<x<x_*$ define, wherever the derivative exists,
\[
 \Lambda(x)=\frac{\dd}{\dd x}
 \{(\zeta(x)-\tau)_++F(\alpha(x))\}.
\]
For $b<x<k-1$, $\Lambda=S(\alpha)\alpha'\leq0$.
For $k-1<x<x_*$, both $(\zeta-\tau)_+$ and $S(\alpha)$ vanish.
Hence $\Lambda\leq0$ wherever it is defined.

Fix $y\in J_\tau$. Its definition gives
$J_\tau\subset[B(1),U(b)]$, and \cref{lem:small-dual-estimates}
gives $b<y_0<0<x_*$. For $b\leq x\leq y_0$, put $r=y-x\geq0$.
Since $x\leq y_0<k-1$, \cref{lem:small-kernel} gives
$S(x)\leq(w-\tau)_+/w$. If $r\geq\tau$, then
$K=(w-r)(1-S)\geq0$; if $0\leq r<\tau<w$, then
$K\geq(w-\tau)r/w\geq0$. If $w\leq\tau$, both $S$ and $K$ vanish.
Thus $K(x,y)\geq0$ for $b\leq x\leq y_0$.
For $y_0\leq x\leq x_*$, \cref{lem:small-dual-estimates} gives
$M:=w-\tau>(k-w)/4>0$.
At a zero of $K(\cdot,y)$ with $y-x<\tau$, put
$R=U(x)-y>0$, $c=k+w$, and $\Delta=c-R=y-B(x)>0$.
Then $S=M/R$. Differentiating the first identity in
\eqref{eq:small-S-gap}, with $\rho=w'=(k-w)/(k+w)$, gives
\[
 cS'=\rho+(1-\rho)S-\Lambda,\qquad
 c\,\partial_x K
 =\Delta\rho(1-S)-(c+R)S+R\Lambda.
\]
Since $\Lambda\leq0$, the latter is strictly negative if
\[
 M>\frac{(k-w)R\Delta}{2k(k+w)+2wR}.
\]
Indeed,
\[
 \frac{(k-w)R\Delta}{2kc+2wR}
 \leq\frac{(k-w)c}{8k}\leq\frac{k-w}{4},
\]
using $R+\Delta=c$, $R\Delta\leq c^2/4$, and $c\leq2k$.
Therefore $\partial_xK<0$ at each such zero.
When $y-x\geq\tau$,
$K=(U-y)(1-S)\geq0$, and $x\leq y-\tau$.
The function $K$ is continuous and piecewise continuously differentiable;
at a transition zero the same strict bound applies to both one-sided
derivatives. Thus $K(\cdot,y)$ cannot change sign from negative
to positive. Hence $G_x(y)$ attains its minimum over $[b,x_*]$ at an
endpoint:
\begin{equation}\label{eq:small-source-minimum}
 G_x(y)\geq\min\{G_b(y),G_{x_*}(y)\},
 \qquad b\leq x\leq x_*,\quad y\in J_\tau.
\end{equation}

\medskip\noindent\emph{4. Nonnegativity of $\partial_yG_x$ for $x\geq x_*$.}
Set $x_c=z(\delta_0)=k+D(x_0)$.
Then $b<x_c<k-1<x_*$.
For $x_c\leq x\leq x_*$ one has
$\alpha(x)\leq x_0$, $w(x)\geq\delta_0$, and $\zeta(x)\leq\tau$.
Consequently, with $H(x)=F(x)-F(x_0)\geq0$,
\[
 H'(x)=S(x)\geq\frac{\delta_0-\tau+H(x)}{k+1}.
\]
The differential inequality and $H(x_c)\geq0$ yield
\[
 H(x_*)\geq(\delta_0-\tau)
 \left[\exp\left(\frac{x_*-x_c}{k+1}\right)-1\right].
\]
Moreover
$x_*-x_c=\alpha_*-x_0+\tau+v_0\geq\tau+v_0-\ell$.
By \eqref{eq:small-activation-bound},
$F(x_*)-F(x_0)\geq w_0-\tau$.
Now $x_*>0>U(b)$, and \eqref{eq:small-target-derivative} gives
\[
 g_x(y_0)=F(x)-F(x_0)-(w_0-\tau)\geq0
 \quad\text{for }x\geq x_*.
\]
Since $F(t)+w(t)-\tau$ decreases along $J_\tau$, the same inequality
holds throughout $J_\tau$.
Thus condition~(i) of \cref{lem:source-comparison} holds, while
\eqref{eq:small-source-minimum} is condition~(ii). Step~1 gives
$G_b\geq0$, and Step~2 locates every possible negative minimum in
$J_\tau$. The comparison lemma therefore yields $G_x\geq0$ for
all $x\in[b,1]$.

The dual inequality now holds on the whole effective domain for each
$h_\tau$. As in the proof of \cref{thm:optimality}, superposition
proves it for smooth $h$ with $h'''\geq0$, because quadratic
terms have the same integral under every admissible coupling.
Uniform approximation extends optimality to $h\in C^1([-3,3])$
with convex derivative.
\end{proof}

\subsection{Two-source comparisons and a counterexample}
\label{sec:gamma-monotonicity}

We state the necessary optimality condition for variations of two
conditional laws and prove that the support of $\pi_k$ is
$\Gamma_k$-left-monotone. Without the constraint, left-monotonicity
singles out the left-curtain coupling and is equivalent to optimality
for martingale Spence--Mirrlees costs
\cite{BeiglbockJuillet,BeiglbockHenryLabordereTouzi}. The constraint
relaxes this condition to $\Gamma_k$-left-monotonicity, which is no
longer sufficient: a discrete counterexample (\cref{ex:counter}) shows
that it does not imply optimality, even together with every two-source
comparison. We take this as an indication that the constrained problem
is genuinely harder than its unconstrained counterpart, for which the
dual inequality of \cref{sec:optimality-limitations} would be
unnecessary.

For an optimal coupling
$\pi(\dd x,\dd y)=\mu(\dd x)\pi_x(\dd y)$,
\cite[Theorem~6.1]{BayraktarZhangZhou} gives a Borel set $\Lambda$
with $\mu(\Lambda)=1$ and the following property. If $x,x'\in\Lambda$
and probability laws $m_x,m_{x'}$ have means $x,x'$, supports in
$[x-k,\infty)$ and $[x'-k,\infty)$, and satisfy
\[
 m_x+m_{x'}=\pi_x+\pi_{x'},
\]
one has
\begin{equation}\label{eq:pairwise-comparison}
 \begin{aligned}
 &\int h(y-x)\pi_x(\dd y)+\int h(y-x')\pi_{x'}(\dd y)\\
 &\hspace{1cm}\leq
   \int h(y-x)m_x(\dd y)+\int h(y-x')m_{x'}(\dd y).
 \end{aligned}
\end{equation}
The uniqueness result in
\cite[Proposition~6.2]{BayraktarZhangZhou} requires a source marginal
supported on two points; two-point conditional laws do not meet that
assumption.

Following \cite[Definition~6.1]{BayraktarZhangZhou}, a set
$S\subset\Gamma_k$ is $\Gamma_k$-left-monotone if it contains no points
\[
 (x,y^-),(x,y^+),(x',y')\in S
\]
such that $x<x'$, $y^-<y'<y^+$, and $y^-\geq x'-k$.
The last inequality is exactly the additional condition ensuring that
\[
 (x,y'),\qquad (x',y^-),\qquad (x',y^+)
\]
all belong to $\Gamma_k$.

\begin{proposition}\label{prop:Gamma-left}
For every $1<k<3$, the support of the coupling $\pi_k$ constructed
above is $\Gamma_k$-left-monotone.
\end{proposition}

\begin{proof}
Fix $x<x'$. Two distinct support values at $x$ must be $D(x)$ and
$U(x)$. If $x'<b$, monotonicity on $[-1,b]$ gives
$D(x')\leq D(x)$ and $U(x')\geq U(x)$, so neither support value at
$x'$ lies strictly between $D(x)$ and $U(x)$. If
$x'\geq b$, the inequality $U(x')\geq U(x)$ again excludes a crossing on
the upper graph.  It must therefore be $y'=D(x')=x'-k$.  If
$D(x)<y'<U(x)$, then $(x',D(x))\notin\Gamma_k$ because
$D(x)<x'-k$. Hence $\Gamma_k$-left-monotonicity holds.
\end{proof}

\begin{example}[Two-source optimality does not imply optimality]\label{ex:counter}
Let $k=3$,
\[
 \mu_0=\frac14(\delta_0+\delta_1+\delta_2+\delta_3)
\]
and
\[
 \nu_0=\frac14\delta_{-2}+\frac16\delta_1+\frac38\delta_2
       +\frac18\delta_4+\frac1{12}\delta_7.
\]
Define the conditional laws
\begin{align*}
 P_0&=\tfrac13\delta_{-2}+\tfrac23\delta_1,&
 P_1&=\tfrac23\delta_{-2}+\tfrac13\delta_7,&
 P_2&=\delta_2,&
 P_3&=\tfrac12\delta_2+\tfrac12\delta_4,\\
 Q_0&=\tfrac12\delta_{-2}+\tfrac12\delta_2,&
 Q_1&=\tfrac12\delta_{-2}+\tfrac12\delta_4,&
 Q_2&=\delta_2,&
 Q_3&=\tfrac23\delta_1+\tfrac13\delta_7.
\end{align*}
Write $P(\dd x,\dd y)=\mu_0(\dd x)P_x(\dd y)$ and similarly for $Q$.
Each displayed conditional law has mean equal to its source, and its
support lies above $y=x-3$. Averaging the four conditional laws
gives $\nu_0$ for both couplings. The only possible strict crossings
in the support of $P$ involve the interval
$(-2,7)$ at source $x=1$ and a later value $2$ or $4$.  Moving its lower
endpoint $-2$ to either later source $x'=2$ or $x'=3$ violates
$-2\geq x'-3$.  Hence the support of $P$ is
$\Gamma_3$-left-monotone.

To verify \eqref{eq:pairwise-comparison} for $P$, let
$h\in C^1(\R)$ with convex derivative. Fix distinct sources $i,j$ and
probability laws $m_i,m_j$ of means $i,j$, supported on
$[i-3,\infty)$ and $[j-3,\infty)$, with $m_i+m_j=P_i+P_j$.
Both laws are supported on $\supp(P_i+P_j)$.
Except for the pair $(0,1)$, the constraints force the original laws:
for $(0,2)$, the largest available target is $2$, so $m_2=\delta_2$;
for $(2,3)$, it is the smallest available target, with the same conclusion.
For $(0,3)$, at most mass $1/2$ is available at $4$, and every other
target is at most $2$. The mean of $m_3$ is therefore at most
$\tfrac12\,4+\tfrac12\,2=3$, and equality forces
$m_3=\tfrac12\delta_2+\tfrac12\delta_4$.
For $(1,2)$ and $(1,3)$, the later source cannot receive $-2$.
Thus $m_1$ retains mass $2/3$ at $-2$. Its remaining mass is $1/3$ and
must have conditional mean $7$, the largest available target, so
$m_1=P_1$.
In each case, preservation of the sum determines the other law.

For $(0,1)$, set $t=m_0(\{7\})$. The mass and mean constraints determine
the other coefficients, and nonnegativity gives precisely
\begin{align*}
 m_0&=(\tfrac13+2t)\delta_{-2}
           +(\tfrac23-3t)\delta_1+t\delta_7,\\
 m_1&=(\tfrac23-2t)\delta_{-2}
           +3t\delta_1+(\tfrac13-t)\delta_7,
 \qquad 0\leq t\leq\tfrac29.
\end{align*}
Put
\[
 g(y):=h(y)-h(y-1)=\int_0^1 h'(y-s)\dd s.
\]
The convexity of $h'$ implies that $g$ is convex. The change in the sum
of the two conditional costs is
\[
 t\,[2g(-2)-3g(1)+g(7)]\geq0,
\]
because $1=\tfrac23(-2)+\tfrac13\,7$.
Comparisons at the same source give equality by linearity.
This proves every instance of \eqref{eq:pairwise-comparison}.

Nevertheless, for $h(z)=z^3$, whose derivative is strictly convex,
\begin{align*}
  \mathbb E^P[h(Y-X)]&=\frac{-2+54+0+0}{4}=13,\\
  \mathbb E^Q[h(Y-X)]&=\frac{0+0+0+16}{4}=4.
\end{align*}
Since $P_2=Q_2$, only the conditional laws at $0,1,3$ change.
Thus a $\Gamma_3$-left-monotone coupling satisfying every two-source
comparison can fail to minimize the cost. The example uses discrete
marginals, distinct from the uniform marginals of the main problem.
\end{example}

\section{Numerical examples}
\label{sec:numerics}

The Python code for the computations in this section is available at
\url{https://github.com/ebayr/constrained-martingale-transport}.
The repository includes a README with dependencies and instructions for
reproducing the numerical values and figure data. Script paths below
refer to this repository.

We evaluate the maps in \cref{sec:construction} for $k=3/2$ and $5/2$,
and compute $V_k(z^3)$ as $k$ varies. The script
\path{generate_data.py} uses the principal Lambert function
for $U|_{[b,1]}$, bisection for the scalar inversions, and composite
Simpson quadrature for expectations. The root of $T(\delta_0)=0$
is bracketed by \cref{lem:small-k-transition}.
The script \path{generate_explanatory_figures.py} produces
\cref{fig:predecessor-geometry,fig:comparison-geometry}.
\Cref{sec:numerics-lp} compares the formulas with a discrete linear program.

\subsection{Two representative constrained couplings}

We evaluate the constructions at $k=3/2$ and $k=5/2$, compare their
support maps with the unconstrained maps, and check the marginal and
martingale conditions numerically.

For $k=3/2$, the scalar construction gives
\[
 \begin{gathered}
 b=-0.5,\qquad
 \delta_b=\KOneFiveDeltaB,\qquad
 \delta_0=\KOneFiveDeltaZero,\\
 x_0=\KOneFiveXZero,\qquad
 \delta_*=\KOneFiveDeltaStar.
 \end{gathered}
\]
The source intervals $[-1,x_0]$, $[x_0,b]$, and $[b,1]$ have
$\mu$-masses $0.0520070$, $0.1979930$, and $0.75$, respectively.
For
$k=5/2$, one has
\[
 a=0,\qquad b=0.5,\qquad \eta_b=\KTwoFiveEtaB.
\]
The intervals $[-1,a]$, $[a,b]$, and $[b,1]$ have $\mu$-masses
$1/2$, $1/4$, and $1/4$, respectively. Optimality follows from
\cref{thm:small-optimality,thm:optimality}.

\Cref{fig:support-examples} compares these maps with
$D_{\lc},U_{\lc}$ from \eqref{eq:ordinary-maps}.

\begin{figure}[t]
\centering
\begin{tikzpicture}
\begin{groupplot}[
 group style={group size=2 by 1,horizontal sep=1.05cm},
 width=0.47\textwidth,
 height=0.38\textwidth,
 xmin=-1.05,xmax=1.05,
 ymin=-2.1,ymax=2.1,
 xlabel={$x$},
 ylabel={$y$},
 tick label style={font=\scriptsize},
 label style={font=\small},
 title style={font=\small},
 grid=major,
 grid style={black!8}
]
\nextgroupplot[title={Optimal coupling, $k=3/2$}]
\addplot[support diagonal] coordinates {(-1,-1) (1,1)};
\addplot[support obstacle] table[x=x,y=B] {support_k15.dat};
\addplot[ordinary lower,domain=-1:1,samples=101] {-x/2-3/2};
\addplot[ordinary upper,domain=-1:1,samples=101] {3*x/2+1/2};
\addplot[constrained lower] table[x=x,y=D] {support_k15.dat};
\addplot[constrained upper] table[x=x,y=U] {support_k15.dat};
\addplot[support seam] coordinates {(\KOneFiveXZero,-2.08) (\KOneFiveXZero,2.08)};
\addplot[support seam] coordinates {(-0.5,-2.08) (-0.5,2.08)};
\node[font=\scriptsize,anchor=north west] at (axis cs:\KOneFiveXZero,2.04) {$x_0$};
\node[font=\scriptsize,anchor=north west] at (axis cs:-0.5,2.04) {$b$};

\nextgroupplot[
 title={Optimal coupling, $k=5/2$},
 legend to name=supportlegend,
 legend columns=3,
 legend style={font=\scriptsize,draw=none,column sep=0.75em},
 legend cell align=left
]
\addplot[support diagonal] coordinates {(-1,-1) (1,1)};
\addlegendentry{$y=x$}
\addplot[support obstacle] table[x=x,y=B] {support_k25.dat};
\addlegendentry{$y=x-k$}
\addplot[ordinary lower,domain=-1:1,samples=101] {-x/2-3/2};
\addlegendentry{$D_{\lc}(x)$}
\addplot[ordinary upper,domain=-1:1,samples=101] {3*x/2+1/2};
\addlegendentry{$U_{\lc}(x)$}
\addplot[constrained lower] table[x=x,y=D] {support_k25.dat};
\addlegendentry{constrained $D(x)$}
\addplot[constrained upper] table[x=x,y=U] {support_k25.dat};
\addlegendentry{constrained $U(x)$}
\addplot[support seam] coordinates {(0,-2.08) (0,2.08)};
\addplot[support seam] coordinates {(0.5,-2.08) (0.5,2.08)};
\node[font=\scriptsize,anchor=north west] at (axis cs:0,2.04) {$a$};
\node[font=\scriptsize,anchor=north west] at (axis cs:0.5,2.04) {$b$};
\end{groupplot}
\end{tikzpicture}
\par\smallskip
\pgfplotslegendfromname{supportlegend}
\caption{The maps $D,U$ for $k=3/2$ and $5/2$, compared with the ordinary
left-curtain maps. Red solid curves with circles represent $D$; blue
dashed curves with squares represent $U$. The black lines are $y=x$
and $y=x-k$. Vertical lines mark the source values where the formulas
change. Optimality holds for every $C^1$ cost with convex derivative
by \cref{thm:small-optimality,thm:optimality}.}
\label{fig:support-examples}
\end{figure}
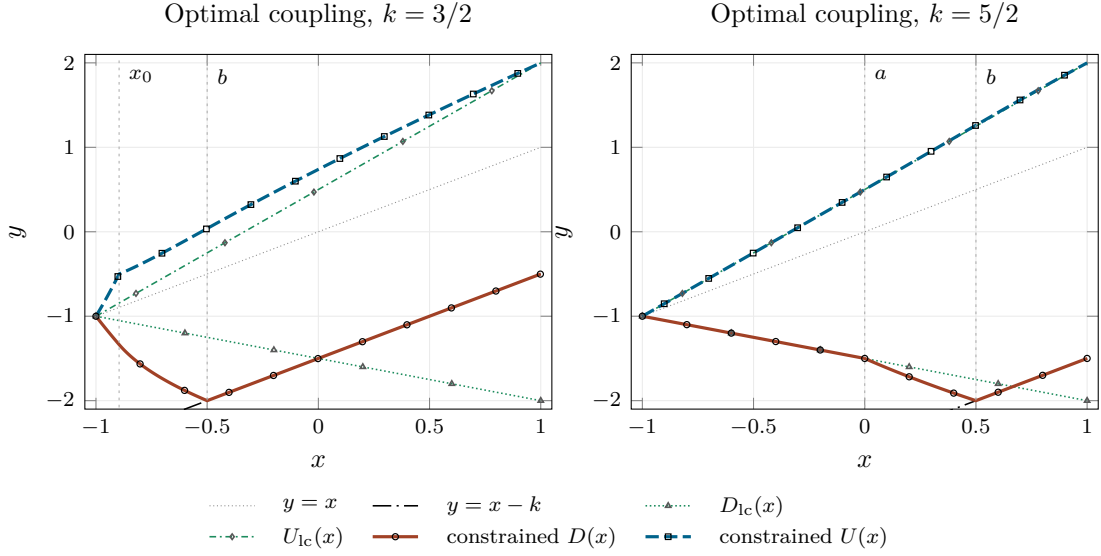

For $k=5/2$, if $y=E(\eta)+\eta-5\in[-2,-3/2]$, it receives mass both
from $D|_{[-1,b]}$ and $D|_{[b,1]}$. Their respective target densities are
\begin{equation}\label{eq:density-decomposition}
 r_{\mathrm m}(y)=\frac{E(\eta)}{4(5-E(\eta))},
 \qquad
 r_{\mathrm b}(y)=\frac{5/2-E(\eta)}{2(5-E(\eta))},
 \qquad r_{\mathrm m}(y)+r_{\mathrm b}(y)=\frac14.
\end{equation}
\Cref{fig:marginal-checks} plots $r_{\mathrm m}$, $r_{\mathrm b}$, and their sum.

\begin{figure}[t]
\centering
\begin{tikzpicture}
\begin{axis}[
 width=0.65\textwidth,
 height=0.32\textwidth,
 xlabel={$y$},
 tick label style={font=\scriptsize},
 label style={font=\small},
 title style={font=\small},
 grid=major,
 grid style={black!8},
 title={Lower-density decomposition, $k=5/2$},
 xmin=-2,xmax=-1.5,
 ymin=0,ymax=0.27,
 ylabel={density contribution},
 legend style={at={(0.5,-0.27)},anchor=north,legend columns=3,
               font=\scriptsize,draw=none,fill=none}
]
\addplot[BrickRed,very thick,solid,mark=square*,mark repeat=50,
         mark size=1.25pt]
 table[x=y,y=middle] {density_decomposition_k25.dat};
\addlegendentry{$r_{\mathrm m}$}
\addplot[MidnightBlue,very thick,densely dashed,mark=*,mark repeat=50,
         mark size=1.25pt]
 table[x=y,y=boundary] {density_decomposition_k25.dat};
\addlegendentry{$r_{\mathrm b}$}
\addplot[black,dotted] table[x=y,y=total] {density_decomposition_k25.dat};
\addlegendentry{sum $=1/4$}
\end{axis}
\end{tikzpicture}
\caption{The target densities contributed by $D|_{[-1,b]}$ and
$D|_{[b,1]}$ on $[-2,-3/2]$ for $k=5/2$. Their sum equals $1/4$
by \eqref{eq:density-decomposition}.}
\label{fig:marginal-checks}
\end{figure}
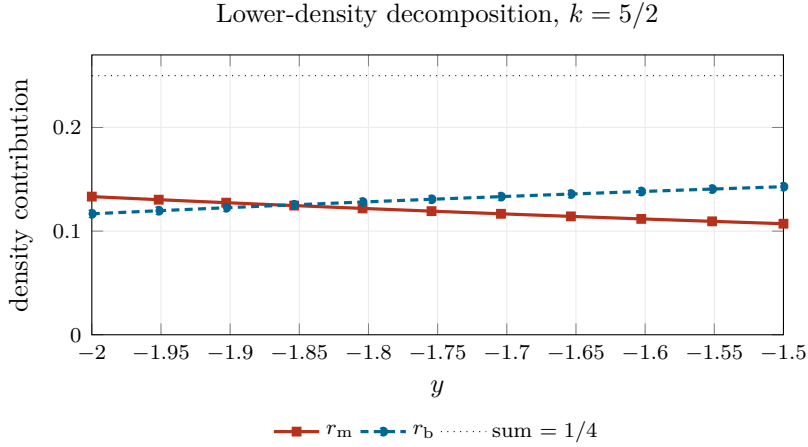

As checks of the implementation, both representative cases have sampled
errors in the martingale and support conditions below $7\cdot10^{-16}$, and computed
second moments agreeing with $\mathbb E[Y^2]=4/3$ to nine decimal places.
Using $20{,}000$ midpoint source nodes gives target CDF errors below
$2\cdot10^{-5}$; with $80{,}000$ nodes and $64$ equal target bins, the
histogram $L^1$ errors are below $2\cdot10^{-4}$.

\subsection{The cost of the constraint}

We quantify the effect of the bound $Y\geq X-k$ by computing the minimum
cubic cost $V_k(z^3)$ and the probability $\pi_k\{Y=X-k\}$ as $k$ varies.

For $h(z)=z^3$, the derivative $h'(z)=3z^2$ is convex, so the values in the
second column of \cref{tab:large-numerics} are optimal for $1<k<3$ by
\cref{thm:small-optimality,thm:optimality}. The endpoints $k=1$ and
$k=3$ follow from \cref{prop:unit-optimality,prop:ordinary}, respectively.
The last column gives the probability of $Y=X-k$,
\[
 \pi_k\{Y=X-k\}=\frac12\int_b^1(1-q_k(x))\dd x.
\]
The row at $k=1$ is exact: the minimum cubic cost is zero and half the
mass is sent to $Y=X-1$. The script
\path{generate_comparison_table.py} generates the remaining rows
directly from the coupling formulas in both constrained regimes.
Doubling the Simpson panels from $2{,}400$ to $4{,}800$ on each source
interval leaves all displayed digits unchanged.

\begin{table}[t]
\centering
\caption{The minimum cubic cost $V_k(z^3)$ and probability
$\pi_k\{Y=X-k\}$, rounded to four decimal places.}
\label{tab:large-numerics}
\small
\begin{tabular}{@{}rrr@{}}
\toprule
$k$ & Minimum cost $V_k(z^3)$ & $\pi_k\{Y=X-k\}$\\
\midrule
1.00 & 0.0000 & 0.5000\\
\ComparisonRows
\bottomrule
\end{tabular}
\end{table}

As $k$ increases to $3$, the value decreases to the unconstrained
value $-3/2$. The feasible coupling $Y=X\pm1$ has cubic cost zero
for every $k\geq1$. \Cref{fig:value-and-mass} also compares
$\pi_k\{X\geq b\}=(3-k)/2$ with $\pi_k\{Y=X-k\}$.

\begin{figure}[t]
\centering
\begin{tikzpicture}
\begin{groupplot}[
 group style={group size=2 by 1,horizontal sep=1.15cm},
 width=0.47\textwidth,
 height=0.36\textwidth,
 xmin=2,xmax=3,
 xlabel={$k$},
 tick label style={font=\scriptsize},
 label style={font=\small},
 title style={font=\small},
 grid=major,
 grid style={black!8},
 legend style={font=\scriptsize,fill=white,fill opacity=0.88,text opacity=1}
]
\nextgroupplot[
 title={Cubic displacement cost},
 ymin=-1.55,ymax=0.08,
 ylabel={$\mathbb E[(Y-X)^3]$},
 legend pos=north east
]
\addplot[MidnightBlue!90!black,very thick,solid,mark=o,mark repeat=10,
         mark size=1.15pt,mark options={solid,draw=black,fill=white}]
 table[x=k,y=value] {cubic_value.dat};
\addlegendentry{$V_k(z^3)$}
\addplot[BrickRed,densely dashed] coordinates {(2,-1.5) (3,-1.5)};
\addlegendentry{unconstrained value}
\addplot[black,dotted] coordinates {(2,0) (3,0)};
\addlegendentry{two-shift value}

\nextgroupplot[
 title={Source mass and mass on the constraint},
 ymin=0,ymax=0.52,
 ylabel={mass},
 legend pos=north east
]
\addplot[BrickRed!85!black,very thick,solid,mark=o,mark repeat=10,
         mark size=1.15pt,mark options={solid,draw=black,fill=white}]
 table[x=k,y=active_source_mass] {cubic_value.dat};
\addlegendentry{$\mathbb P(X\geq b)$}
\addplot[MidnightBlue!90!black,very thick,densely dashed,mark=square,
         mark repeat=10,mark phase=5,mark size=1.05pt,
         mark options={solid,draw=black,fill=white}]
 table[x=k,y=boundary_mass] {cubic_value.dat};
\addlegendentry{$\mathbb P(Y=X-k)$}
\end{groupplot}
\end{tikzpicture}
\caption{Left: $V_k(z^3)$ for $2\leq k\leq3$, with the unconstrained value
$-3/2$ and the two-shift value $0$. Right: the probabilities
$\pi_k\{X\geq b\}$ and $\pi_k\{Y=X-k\}$. The former includes
transport to both $D(x)$ and $U(x)$ for $x\in[b,1]$. Both vanish
at $k=3$.}
\label{fig:value-and-mass}
\end{figure}
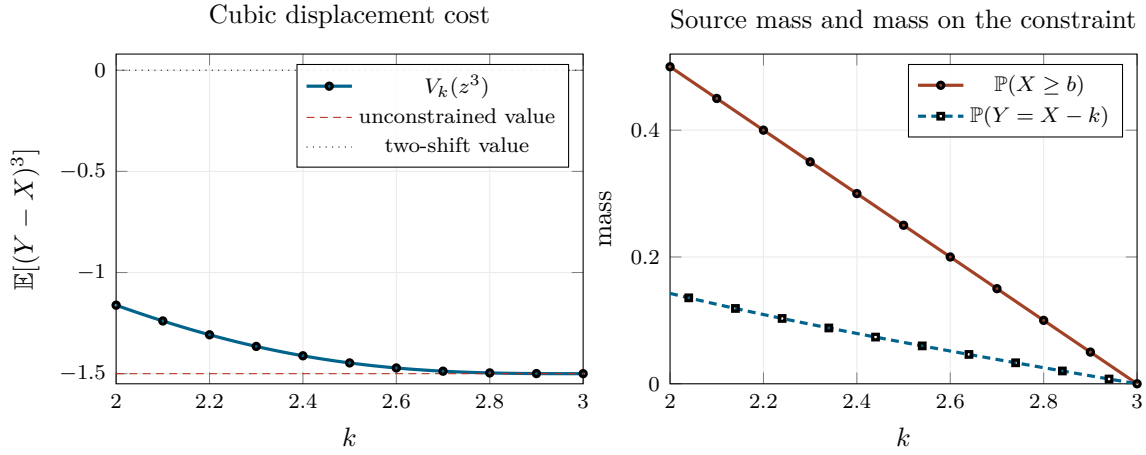

\subsection{Independent finite-dimensional optimization}
\label{sec:numerics-lp}

We compare the cubic cost and support of $\pi_{5/2}$ and $\pi_{3/2}$
with discrete optimizers computed without prescribing the support graphs.

For a positive integer $N$, discretize the marginals on the grids
\[
 x_i=-1+\left(i+\frac12\right)\frac2N,
 \qquad
 y_j=-2+\left(j+\frac12\right)\frac2N
\]
for $i=0,\ldots,N-1$ and $j=0,\ldots,2N-1$, with masses $1/N$ and
$1/(2N)$, respectively.  For a given $k$, minimize
\begin{equation*}
 \sum_{i,j}(y_j-x_i)^3\pi_{ij}
\end{equation*}
subject to
\[
 \sum_j\pi_{ij}=\frac1N,\qquad
 \sum_i\pi_{ij}=\frac1{2N},\qquad
 \sum_j(y_j-x_i)\pi_{ij}=0,\qquad
 \pi_{ij}=0\ \text{if }y_j<x_i-k.
\]
We also impose $\pi_{ij}\geq0$. These constraints fix both marginals
and the conditional means exactly. The script
\path{discrete_lp.py} solves the problem with SciPy's HiGHS
interface. \Cref{fig:lp-validation} compares the optimizer for $N=80$
with the graphs of $D,U$. From $N=20$ to $N=120$, the objective error
relative to
$V_{5/2}(z^3)\approx\CubicReferenceKTwoFive$ decreases from
$5.820\cdot10^{-3}$ to $1.680\cdot10^{-4}$, while the mean distance to the
two graphs,
\[
 \sum_{ij}\pi_{ij}\min\{|y_j-D(x_i)|,|y_j-U(x_i)|\},
\]
decreases from $3.607\cdot10^{-2}$ to $6.060\cdot10^{-3}$ on the four
grids $N=20,40,80,120$.
The same experiment at $k=3/2$, where $\pi_k$ is given by the scalar
construction of \cref{sec:small-k}, gives objective errors relative to
$V_{3/2}(z^3)\approx\CubicReferenceKOneFive$ decreasing from
$3.744\cdot10^{-3}$ to $1.049\cdot10^{-4}$, and mean distances to the
two graphs decreasing from $3.236\cdot10^{-2}$ to $5.499\cdot10^{-3}$,
on the same four grids. Across all eight linear programs, the maximum
row, column, and martingale residuals are below $3\cdot10^{-16}$.
These are floating-point residuals; the feasibility and optimality
proofs are independent of the discretization.

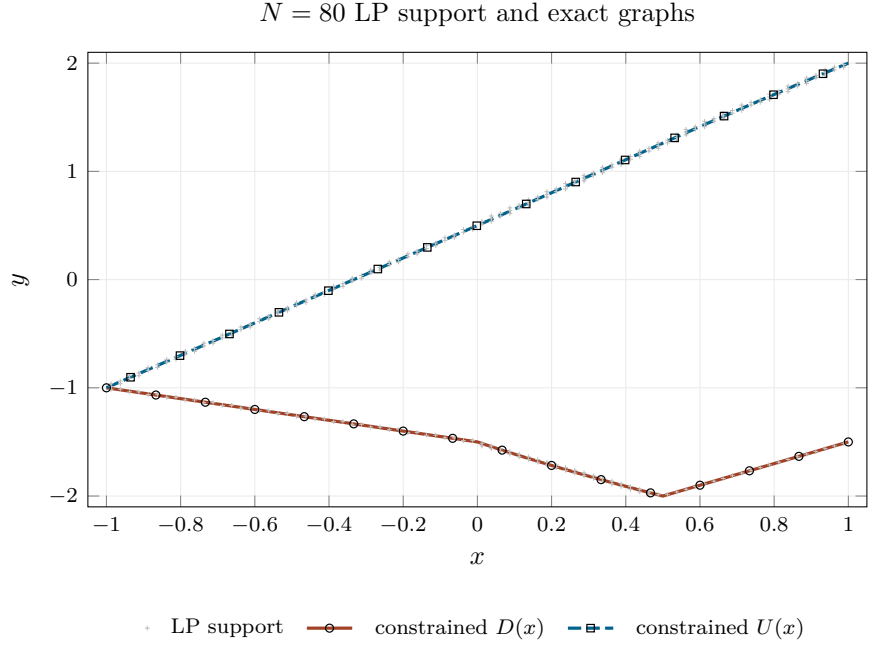
\begin{figure}[t]
\centering
\begin{tikzpicture}
\begin{axis}[
 width=0.72\textwidth,
 height=0.46\textwidth,
 tick label style={font=\scriptsize},
 label style={font=\small},
 title style={font=\small},
 grid=major,
 grid style={black!8},
 title={$N=80$ LP support and exact graphs},
 xmin=-1.05,xmax=1.05,
 ymin=-2.1,ymax=2.1,
 xlabel={$x$},ylabel={$y$},
 legend style={at={(0.5,-0.22)},anchor=north,legend columns=3,
               font=\scriptsize,draw=none,fill=none,column sep=6pt},
 legend cell align=left
]
\addplot[only marks,mark=+,mark size=0.9pt,
         mark options={draw=black!25,line width=0.25pt}]
 table[x=x,y=y] {lp_support_k25.dat};
\addlegendentry{LP support}
\addplot[constrained lower,line width=1.2pt,mark size=1.5pt,mark repeat=80]
 table[x=x,y=D] {support_k25.dat};
\addlegendentry{constrained $D(x)$}
\addplot[constrained upper,line width=1.2pt,mark size=1.4pt,mark repeat=80,
         mark phase=40]
 table[x=x,y=U] {support_k25.dat};
\addlegendentry{constrained $U(x)$}
\end{axis}
\end{tikzpicture}
\caption{The discrete optimizer for $k=5/2$, $h(z)=z^3$, and $N=80$
(gray crosses), together with $D$ (red, solid, circles) and $U$
(blue, dashed, squares). The lower map satisfies $D(x)=x-k$
for $x\geq b=1/2$.}
\label{fig:lp-validation}
\end{figure}

\section{Conclusion}\label{sec:discussion}

For every $k\geq1$, the constructed coupling minimizes
\eqref{eq:primal} for all $C^1$ costs with convex derivative.
The construction changes at $k=1,2,3$, as summarized in
\cref{tab:regimes}. For $1<k<3$, $D(x)=x-k$ on $[b,1]$,
where $b=k-2$. The maps on $[-1,b]$ are determined by the mass and
first-moment equations for the residual measures.

\begin{table}[ht]
\centering
\caption{Parameter regimes for the common optimizer. For $1<k<3$,
$b=k-2$; for $2\leq k<3$, $a=2k-5$.}
\label{tab:regimes}
\renewcommand{\arraystretch}{1.18}
\begin{tabular}{@{}p{0.15\textwidth}p{0.52\textwidth}p{0.25\textwidth}@{}}
\toprule
Range of $k$ & Maps and conditional laws & Result\\
\midrule
$0<k<1$ & $Y\geq-1-k>-2$ contradicts the target marginal
         & No feasible coupling\\
$k=1$ & Two shifts $Y=X\pm1$, with equal weights & Optimal\\
$1<k<2$ & Scalar construction on $[-1,x_0]$; \eqref{eq:middle-delta} on $[x_0,b]$; $D(x)=x-k$ on $[b,1]$
          & Optimal\\
$2\leq k<3$ & $D_{\lc},U_{\lc}$ on $[-1,a]$; \eqref{eq:x-middle}--\eqref{eq:U-middle} on $[a,b]$; $D(x)=x-k$ on $[b,1]$
             & Optimal\\
$k\geq3$ & Ordinary left-curtain maps
          & Unconstrained optimizer\\
\bottomrule
\end{tabular}
\end{table}

Feasibility is characterized by $\Gamma_k$-convex order.
Optimality is proved by the dual inequality and uniform approximation.
\Cref{ex:counter} shows that the two-source comparison in
\cite[Theorem~6.1]{BayraktarZhangZhou} is insufficient, even for a
$\Gamma_k$-left-monotone coupling.

\begin{remark}[Other marginals]\label{rem:other-marginals}
Three thresholds explain the regimes in the uniform example. The
ordinary left-curtain coupling first becomes feasible when $k$ reaches
its largest downward displacement, here $3$. The left endpoint of the
target support gives $b=-2+k$: sources $x>b$ cannot reach $-2$, and our
construction has $D(x)=x-k$ precisely on $[b,1]$. Finally, the boundary
$x-k$ stays at or below the common initial target $D(-1)=U(-1)=-1$ for all
$x\leq1$ exactly when $k\geq2$. Below this threshold, the predecessor
of a boundary point may lie on the upper graph.

For general compactly supported marginals in convex order, the
ordinary left-curtain coupling satisfies the constraint exactly when
$k\geq\operatorname*{ess\,sup}_{\pi_{\lc}}(x-y)$.
The other two thresholds depend on the geometry of its support and
need not describe the constrained optimizer. In particular, the fact
that a source cannot reach the left endpoint of the target support
does not by itself force the lower constraint to bind.

A conditional version of the construction is nevertheless available.
Suppose that $\mu,\nu$ have densities $\mu',\nu'$ and supports
$[x_-,x_+]$, $[y_-,y_+]$, respectively. Assume
$b:=y_-+k\in(x_-,x_+)$, and seek an increasing $C^1$ upper map with
$U(x)>x$, taking values in $[y_-,y_+]$, that satisfies
\[
 \nu'(U(x))U'(x)
 =\mu'(x)\frac{k}{U(x)-x+k},\qquad U(x_+)=y_+.
\]
Together with $D(x)=x-k$ and the martingale weights, such a map
defines a transport from $[b,x_+]$, provided $x-k\in[y_-,y_+]$
there. Subtract its full target marginal from $\nu$. If the remainder
is nonnegative, is in convex order above
$\mu|_{[x_-,b]}$, and its left-curtain coupling satisfies $y\geq x-k$,
adding the two transports gives an admissible coupling.

These are feasibility conditions, not an optimality theorem for
general marginals. Extending the dual argument requires further
control of the support maps, their inverse branches, and the upper
displacement. In the uniform example, the explicit formulas provide
this control through \cref{lem:hinge,lem:small-dual-estimates}.
Even nonnegativity of the residual target can fail: the lower branch
alone contributes density
$\mu'(y+k)p(y+k)$ on $[y_-,x_+-k]$, where
$p(x)=(U(x)-x)/(U(x)-x+k)$. If this exceeds $\nu'(y)$ on a set of
positive Lebesgue measure, the proposed construction is infeasible.
\end{remark}

\bibliographystyle{amsplain}
\bibliography{reference}

\end{document}